\documentclass[11pt, reqno, english]{amsart}

\usepackage[utf8]{inputenc}
\usepackage[T1]{fontenc}
\usepackage{amsmath,amsthm}
\usepackage{amsfonts,amssymb}
\usepackage{url}
\usepackage{mathtools}  
\usepackage[colorlinks=true,urlcolor=blue,linkcolor=blue,citecolor=blue]{hyperref}
\usepackage{enumerate, paralist}
\usepackage{tikz,tikz-cd}
\usetikzlibrary{calc,arrows}
\usetikzlibrary{math} 
\usetikzlibrary{arrows,decorations,backgrounds,calc,math}
\tikzstyle{map}=[->,semithick]
\tikzstyle{arc}=[bend left,->,semithick]
\tikzstyle{rinclusion}=[right hook->,semithick]
\tikzstyle{linclusion}=[left hook->,semithick]

\usepackage[left=1in,right=1in,top=1in,bottom=1in]{geometry}
\usepackage{xcolor}
\usepackage{hyperref}
\usepackage{float}
\usepackage{subfigure}

\newtheoremstyle{myremark} 
{7pt}                    
{7pt}                    
{}  	                 
{}                           
{\bf}       	         
{.}                          
{.5em}                       
{}  

\theoremstyle{plain}
\newtheorem{lemma}{Lemma}[section]
\newtheorem{theorem}[lemma]{Theorem}
\newtheorem{corollary}[lemma]{Corollary}

\newtheorem{conjecture}[lemma]{Conjecture}
\newtheorem{claim}[lemma]{Claim}
\theoremstyle{definition}

\newtheorem*{definition-cnk}{Definition~\ref{def:cnk}}
\newtheorem{question}{Question}
\newtheorem*{question-motivating}{Motivating Question}


\newcounter{parentnumber}

\theoremstyle{myremark}

\newtheorem{example}[lemma]{Example}

\newcommand{\R}{\mathbb{R}}
\newcommand{\Z}{\mathbb{Z}}

\newcommand{\RP}{\ensuremath{\mathbb{R}\mathrm{P}}}
\newcommand{\CP}{\ensuremath{\mathbb{C}\mathrm{P}}}

\newcommand{\diam}{\mathrm{diam}}

\newcommand{\conn}{\mathrm{conn}}

\newcommand{\st}{\mathrm{st}}

\newcommand{\so}{\mathrm{SO}}
\newcommand{\gh}{\mathrm{GH}}

\newcommand{\coind}{\mathrm{coind}}

\newcommand{\cov}{\mathrm{cov}}
\newcommand{\numCover}{\mathrm{numCover}}
\newcommand{\pack}{\mathrm{pack}}

\newcommand{\vr}[2]{\mathrm{VR}(#1;#2)}

\newcommand{\vrleq}[2]{\mathrm{VR}_\le(#1;#2)}
\newcommand{\vrm}[2]{\mathrm{VR}^\mathrm{m}(#1;#2)}

\newcommand{\vrmleq}[2]{\mathrm{VR}^\mathrm{m}_\le(#1;#2)}

\usepackage{accents}
\usepackage[normalem]{ulem}
\usepackage{wasysym}
\usepackage{charter}
\usepackage{euler}
\usepackage[T1]{fontenc}
\usepackage{thmtools}
\usepackage{thm-restate}

\begin{document}

\title{
The connectivity of Vietoris--Rips complexes of spheres
}

\author{Henry Adams}
\address[HA]{Department of Mathematics, University of Florida, USA}
\email{henry.adams@ufl.edu}

\author{Johnathan Bush}
\address[JB]{Department of Mathematics and Statistics, James Madison University, USA}
\email{bush3je@jmu.edu}

\author{\v{Z}iga Virk}
\address[ZV]{University of Ljubljana and Institute IMFM, Slovenia}
\email{ziga.virk@fri.uni-lj.si}

\begin{abstract}
We survey what is known and unknown about Vietoris--Rips complexes and thickenings of spheres.
Afterwards, we show how to control the homotopy connectivity of Vietoris--Rips complexes of spheres in terms of coverings of spheres and projective spaces.
Let $S^n$ be the $n$-sphere with the geodesic metric, and of diameter $\pi$, and let $\delta > 0$.
Suppose that the first nontrivial homotopy group of the Vietoris--Rips complex $\vr{S^n}{\pi-\delta}$ of the $n$-sphere at scale $\pi-\delta$ occurs in dimension $k$, i.e., suppose that the connectivity is $k-1$.
Then $\cov_{S^n}(2k+2) \le \delta < 2\cdot \cov_{\RP^n}(k)$.
In other words, there exist $2k+2$ balls of radius $\delta$ that cover $S^n$, and no set of $k$ balls of radius $\frac{\delta}{2}$ cover the projective space $\RP^n$.
As a corollary, the homotopy type of $\vr{S^n}{r}$ changes infinitely many times as the scale $r$ increases.
\end{abstract}

\keywords{Vietoris--Rips complexes, spheres, homotopy connectivity, covering radii}

\subjclass{55N31, 55Q52, 52C17}

\maketitle


\section{Introduction}

\subsection*{Overview of the main result}
We study the homotopy connectivity of Vietoris--Rips complexes of spheres.
Let $S^n$ be the $n$-sphere equipped with the geodesic metric, and of diameter $\pi$.
Let the projective space $\RP^n$ be equipped with the quotient metric, so that $\RP^n$ has diameter $\frac{\pi}{2}$.
The \emph{Vietoris--Rips simplicial complex $\vr{S^n}{r}$} has $S^n$ as its vertex set, with a finite set $\sigma\subseteq S^n$ as a simplex if $\diam(\sigma)<r$.
The \emph{homotopy connectivity} $\conn(\vr{S^n}{r})$ is the largest integer $j$ such that the homotopy groups $\pi_i(\vr{S^n}{r})$ are trivial for $i\le j$.
For $X$ a metric space and $k$ a positive integer, let the \emph{covering radius} $\cov_X(k)$ be the infimal scale $r\ge 0$ such that there exist $k$ closed balls of radius $r$ in $X$ that cover all of $X$.
Our main theorem provides bounds on the homotopy connectivity of Vietoris--Rips complexes of spheres in terms of covering radii of spheres and of projective spaces.

\begin{theorem}
\label{thm:main}
For $n\ge 1$ and $\delta>0$, if $\conn(\vr{S^n}{\pi-\delta})=k-1$ then
\[\cov_{S^n}(2k+2) \le \delta < 2\cdot \cov_{\RP^n}(k).\]
\end{theorem}

Here, $\cov_{S^n}(2k+2)$ is the smallest radius such that $2k+2$ balls of that radius cover $S^n$, and $\cov_{\RP^n}(k)$ is the smallest radius such that $k$ balls of that radius cover the projective space $\RP^n$ of diameter $\frac{\pi}{2}$.
So, Theorem~\ref{thm:main} says that if $\conn(\vr{S^n}{\pi-\delta})=k-1$, then (i) there exist $2k+2$ balls of radius $\delta$ that can cover $S^n$, and (ii) no $k$ balls of radius $\frac{\delta}{2}$ can cover all of $\RP^n$.
Lifting from $\RP^n$ to $S^n$, (ii) is the same as saying there is no set $X$ of $k$ points in $S^n$ such that the $2k$ balls of radius $\frac{\delta}{2}$ centered at the points in $X\cup(-X)$ cover $S^n$.

We remark that $\vr{S^n}{r}$ is contractible if $r\ge \pi$.
See Figures~\ref{fig:s1_intervals-intro},~\ref{fig:S1-intervals}, and~\ref{fig:S2-intervals} for illustrations of Theorem~\ref{thm:main} in the case $n=1$ and $n=2$.

\begin{figure}[h]
\includegraphics[width=\linewidth]{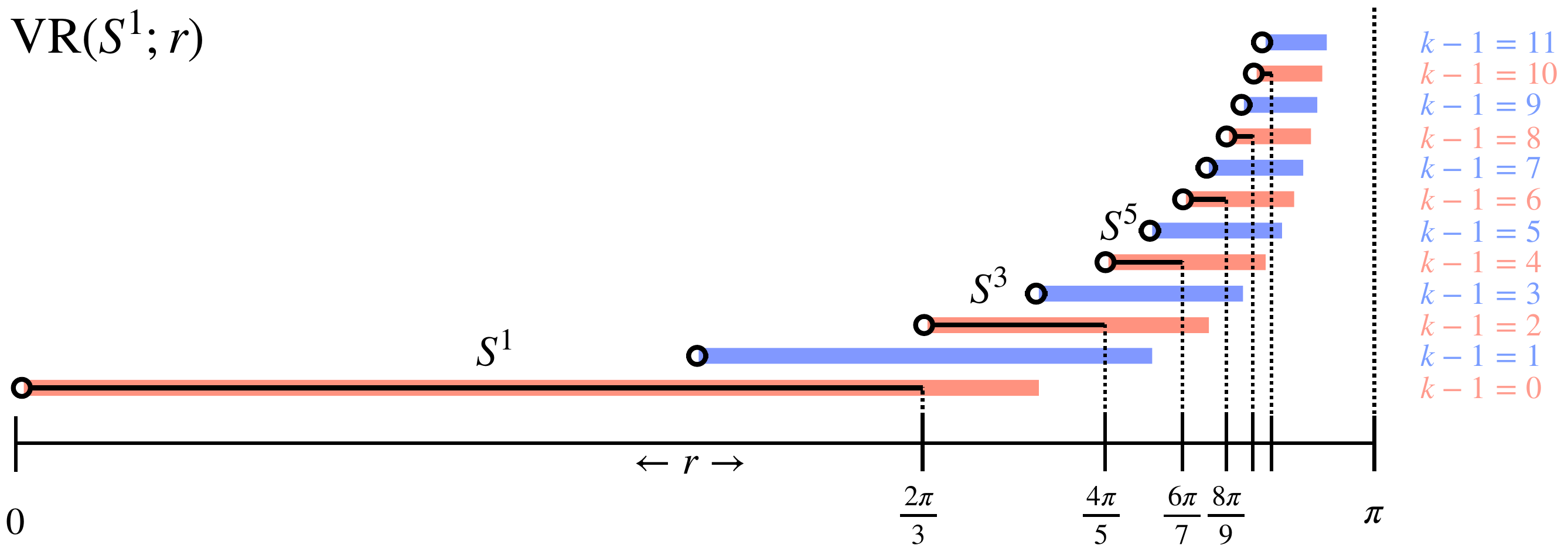}
\caption{The homotopy types of $\vr{S^1}{r}$ as $r$ varies~\cite{AA-VRS1} are indicated by black bars (cf.\ Figure~\ref{fig:circle-barcodes}).
Theorem~\ref{thm:main} gives intervals where $\vr{S^1}{r}$ may have connectivity $k-1$, which are indicated by colored bars (cf.\ Figure~\ref{fig:S1-intervals}).
Left endpoints of red bars are tight. 
All intervals have open left-endpoints and closed right-endpoints.}
\label{fig:s1_intervals-intro}
\end{figure}

Another way to present our main result is as follows.
For $X$ a metric space and $\delta>0$, let the \emph{covering number} $\numCover_X(\delta)$ be the minimum number of closed balls of radius $\delta$ that are needed in order to cover $X$.
Theorem~\ref{thm:main} bounds the homotopy connectivity of Vietoris--Rips complexes of spheres in terms of covering numbers of spheres and of projective spaces:
\begin{corollary}
\label{cor:main1}
For $n\ge 1$ and $\delta>0$,
\[\tfrac{1}{2}\numCover_{S^n}(\delta)-2 \le \conn(\vr{S^n}{\pi-\delta}) \le \numCover_{\RP^n}(\tfrac{\delta}{2})-2.\]
\end{corollary}

As a consequence, we obtain the following corollary.

\begin{corollary}
\label{cor:main2}
For $n\ge 1$ the homotopy type of $\vr{S^n}{r}$ changes infinitely many times as $r$ increases.
Indeed, for any $\varepsilon>0$, there are infinitely many changes to the homotopy type of $\vr{S^n}{r}$ as $r$ varies over the interval $(\pi-\varepsilon,\pi)$.
\end{corollary}

We conjecture that the number of different homotopy types is countably infinite.

\begin{conjecture}
\label{conj:countable}
We conjecture that there are a countable number of critical scales $0=r_0 < r_1 < r_2 < r_ 3 < \ldots < \pi$
such that for all $r_i<r<r'<r_{i+1}$, the inclusion $\vr{S^n}{r}\hookrightarrow\vr{S^n}{r'}$ is a homotopy equivalence.
\end{conjecture}

The proof of Theorem~\ref{thm:main} relies on several prior results.
Indeed, the first inequality in Theorem~\ref{thm:main} follows from a result by Meshulam~\cite{meshulam2001clique} which lower bounds the connectivity of clique complexes; see also Kahle~\cite{Kahle2009} or an appendix by Barmak~\cite{farber2023large}.
Though we have been interested in the homotopy types of Vietoris--Rips complexes of spheres for a decade, only in November 2023 did we learn about Meshulam's techniques, from a recent paper by Bendersky and Grbi\'{c}~\cite{bendersky2023connectivity}.
The second inequality follows from Adams, Bush, and Frick~\cite{ABF2}, who upper bound the $\Z/2$ topology of Vietoris--Rips \emph{metric thickenings}; these upper bounds also apply to Vietoris--Rips simplicial complexes.
We believe that determining the homotopy types of Vietoris--Rips complexes of spheres is an important problem, and that the bounds on connectivity given in Theorem~\ref{thm:main} are a strong initial step in this direction.

\subsection*{Why are the homotopy types of Vietoris--Rips complexes of spheres $\vr{S^n}{r}$ important?}
We provide three main reasons.

\emph{First, they arise in applications of topology to data.}
One of the main applications of topology is to the analysis of data, as the shape of data can reflect important patterns within~\cite{Carlsson2009}.
A popular way for data scientists to approximate the shape of a dataset is by thickening it, and then computing peristent homology~\cite{EdelsbrunnerHarer,edelsbrunner2000topological,robins2000computational}.
A common way to thicken a dataset is to compute a Vietoris--Rips complex~\cite{Vietoris27}.
We do not claim that a better knowledge of Vietoris--Rips complexes of spheres will lead to new data science techniques.
But, an important role for mathematicians is to provide a foundational understanding of existing data science techniques, and we do not fully understand applications of persistent homology to data without understanding Vietoris--Rips complexes of spheres.
A dataset would be deemed to be ``simple'' if it were sampled from a distribution whose support was an $n$-sphere, yet we do not know the Vietoris--Rips persistent homology of such datasets for $n\ge 2$.

\emph{Second, the homotopy types of $\vr{S^n}{r}$ have mathematical applications.}
For $X$ a metric space and $r\ge 0$, the \emph{Vietoris--Rips simplicial complex $\vr{X}{r}$} has $X$ as its vertex set, with a finite set $\sigma\subseteq X$ as a simplex if $\diam(\sigma)<r$.
Thus, $\vr{X}{r}$ can be thought of as a natural way to incorporate \emph{all} finite subsets of $X$ of diameter bounded by $r$ into a single parameter space or conformation space.
The homotopy type of this Vietoris--Rips space $\vr{X}{r}$ has been used to provide generalizations of the Borsuk--Ulam theorem~\cite{ABF,ABF2}, and in order to provide quantitative bounds on Gromov--Hausdorff distances between metric spaces~\cite{ChazalDeSilvaOudot2014,chazal2009gromov,GH-BU-VR,HvsGH}.

\emph{Third, understanding $\vr{S^n}{r}$ requires techniques combining mathematical areas.}
It appears to be a difficult problem to determine the homotopy types of Vietoris--Rips complexes of spheres.
A variety of tools have proven useful, including combinatorial topology, metric geometry, and geometric group theory.
We predict that new techniques will be invented in order to understand the homotopy types of Vietoris--Rips complexes of spheres, and that these new techniques will combine ideas from different fields of mathematics in creative and valuable ways.

\subsection*{Outline}
In this paper, we provide bounds on the homotopy connectivity of Vietoris--Rips complexes of spheres, i.e.\ on the dimension of the first nontrivial homotopy group.
Our goal is to explain these bounds in an accessible and self-contained manner.
Along the way, we survey the history of the problem.
We begin in Section~\ref{sec:related} with a survey of work related to the homotopy types of Vietoris--Rips complexes of spheres.
Some portions of this section are forward-looking, including the explanation of the need for a Morse theory and a Morse--Bott theory; we hope this related work section is a valuable contribution of our paper.
Section~\ref{sec:notation} contains the notation we need for our main results.
In Section~\ref{sec:proof} we first prove Theorem~\ref{thm:main}.
Our proof relies on two steps: lower bounds on the connectivity of clique complexes due to Meshulam, Kahle, or Barmak~\cite{meshulam2001clique,Kahle2009,farber2023large}, and upper bounds on the connectivity of Vietoris--Rips thickenings of spheres due to Adams, Bush, and Frick~\cite{ABF2}.
Afterwards, we also prove Corollaries~\ref{cor:main1} and~\ref{cor:main2} in Section~\ref{sec:proof}.
In Section~\ref{sec:examples} we provide examples of the known connectivity bounds on $\vr{S^n}{r}$, relying on the values of covering radii of spheres and of projective spaces.
We conclude in Section~\ref{sec:conclusion} by sharing a list of open questions.

\section{Related Work}
\label{sec:related}

We survey related work on Vietoris--Rips complexes and thickenings, focusing on recent work in applied topology.

\subsection{Vietoris and Rips}
Vietoris--Rips complexes were invented by 
Vietoris in a cohomology theory for metric spaces~\cite{Vietoris27,lefschetz1942algebraic}.
Independently, they were introduced by 
Rips in geometric group theory as a natural way to thicken a space.
Indeed, Rips used these complexes to show that torsion-free hyperbolic groups have finite-dimensional Eilenberg--MacLane spaces~\cite[III.G.3, Theorem~3.21]{bridson2011metric}.

\subsection{Coarse geometry}
See~\cite{berestovskii2007uniform,brodskiy2013rips,cencelj2012combinatorial,conant2014discrete,plaut2013discrete,roe1993coarse,WilkinsThesis} for connections of Vietoris--Rips complexes to coarse geometry, and~\cite{bartholdi2012hodge} for a connection to Hodge theory.

\subsection{Hausmann and Latschev}
In~\cite{Hausmann1995}, Hausmann proves that for $M$ a manifold and scale parameter $r$ sufficiently small compared to the curvature of $M$, the Vietoris--Rips complex $\vr{M}{r}$ is homotopy equivalent to $M$.
This is not an easy theorem to prove.
Indeed, the inclusion map $M\hookrightarrow \vr{M}{r}$ is not continuous, since the vertex set of a simplicial complex is equipped with the discrete metric.
Hausmann's result can be thought of as an analogue of the nerve theorem~\cite{Borsuk1948} for Vietoris--Rips complexes.
Latschev~\cite{Latschev2001} extends Hausmann's result to nearby metric spaces: if a metric space $X$ is sufficiently close to $M$ in the Gromov--Hausdorff distance, and the scale $r$ is large enough compared to the density of the noisy sample $X$ but small enough compared to the curvature of the manifold $M$, then $\vr{X}{r}$ is also homotopy equivalent to $M$.

\subsection{Persistent homology}
The results of Hausmann and Latschev are part of the motivation for why Vietoris--Rips complexes have become commonly used tools in applied and computational topology.
Large sets of high-dimensional data are common in most branches of science, and their shapes reflect important patterns within.
One frequently used tool to recover the shape of data is persistent homology~\cite{Carlsson2009}.
How can we recover the shape of a dataset $X$ sampled from an unknown metric space $M$?
Latschev's result shows we can do this by computing a Vietoris--Rips complex on $X$ with a scale parameter $r$ sufficiently small compared to the geometry of $M$.
Since we do not know \emph{a priori} how to choose the scale $r$, the idea of persistent homology is to compute the homology of the Vietoris--Rips complex of a dataset $X$ over a large range of scale parameters $r$ and to trust those topological features which persist.
As the main idea of persistence is to allow $r$ to vary, the assumption that scale $r$ is kept sufficiently small fails in practice:
data scientists instead let the scale parameter $r$ in the Vietoris--Rips complexes $\vr{X}{r}$ vary from zero to very large.
The persistent homology of Vietoris--Rips complexes is being computed frequently, and there is efficient software designed to do this~\cite{bauer2021ripser,zhang2020gpu,GlissePritam2022}, even though we do not yet have a mathematical understanding of how these simplicial complexes behave at large scales.

Suppose that we are in the particularly nice setting where our dataset is sampled from a manifold $M$.
This is the \emph{manifold hypothesis} which underlies many techniques in manifold learning~\cite{lee2007nonlinear}.
Suppose that we sample more and more points from $M$; let $X_n$ denote a sample of $n$ points taken from some probability distribution whose support is $M$.
As $n$ goes to infinity, the Hausdorff distance between $X_n$ and $M$ goes to zero.
Therefore, the stability of persistent homology~\cite{ChazalDeSilvaOudot2014,chazal2009gromov,cohen2007stability} implies that the persistent homology of the Vietoris--Rips complexes of $X_n$ converges to the persistent homology of the Vietoris--Rips complexes of $M$.
But, what is the persistent homology of the Vietoris--Rips complexes of a manifold?
We still have much to learn about this question.
Indeed, to the best of our knowledge, the point and the circle are the only connected compact Riemannian manifolds without boundary whose Vietoris--Rips persistent homology is known at all scales, or whose Vietoris--Rips complexes are known up to homotopy type at all scales.
In other words, even when we are presented with extremely clean data satisfying the manifold hypothesis, we cannot fully describe what the persistent homology of $X_n$ converges to as more points are drawn, without first better understanding the behavior of Vietoris--Rips complexes of manifolds.

\subsection{Vietoris--Rips complexes of the circle}
As mentioned above, the homotopy types and persistent homology of the Vietoris--Rips complexes of the circle are known.
For $k\ge 0$ an integer and $\frac{2\pi k}{2k+1} < r \le \frac{2\pi (k+1)}{2k+3}$, we have a homotopy equivalence $\vr{S^1}{r}\simeq S^{2k+1}$~\cite{AA-VRS1}; see Figure~\ref{fig:circle-barcodes}.
Therefore $\conn(\vr{S^1}{r})=2k$ for $\frac{2\pi k}{2k+1} < r \le \frac{2\pi (k+1)}{2k+3}$.
The intuition behind the first change in homotopy type is as follows.
For $0 < r \le \frac{2\pi}{3}$, even though the Vietoris--Rips complex $\vr{S^1}{r}$ is of unbounded dimension, it collapses down onto the underlying circle, providing a homotopy equivalence $\vr{S^1}{r}\simeq S^1$.
The first critical scale occurs at $r=\frac{2\pi}{3}$, after which an inscribed equilateral triangle appears to fill in the circle $S^1$.
However, due to the symmetry of the circle, it is actually the case that an entire \emph{circle's worth} of equilateral triangles appears, thus filling in the original $S^1$ in a \emph{circle's worth} of ways.
This yields the $3$-sphere $S^3$ up to homotopy type.
arising via the standard genus one Heegaard decomposition $S^3=S^1\times D^2\cup_{S^1\times S^1}D^2\times S^1$; see Figure~\ref{fig:circle-critical}(left).
Once $r$ exceeds $\frac{4\pi}{5}$ a circle's worth of 4-simplices are attached to $S^3$ to produce $S^5$, and once $r$ exceeds $\frac{6\pi}{7}$ a circle's worth of 6-simplices are glued to $S^5$ to produce the $S^7$, etc.
Inductively, the Vietoris--Rips complex of the circle is homotopy equivalent to a $(k+1)$-fold join of circles, which gives a $(2k+1)$-dimensional sphere:\footnote{
A more accurate description is in fact $\vr{S^1}{r}\simeq S^1 * \tfrac{S^1}{\Z/3} * \tfrac{S^1}{\Z/5} * \ldots * \tfrac{S^1}{\Z/(2k+1)} = S^{2k+1}$; see~\cite{PersistentEquivariantCohomology}.
Indeed, after scale $r=\frac{2\pi k}{2k+1}$ a circle's worth of critical $2k$-dimensional simplices appear.
This circle is best described as $\tfrac{S^1}{\Z/(2k+1)}$  due to the rotational symmetries of $(2k+1)$-evenly spaced points on the circle.
The new homotopy type of the Vietoris--Rips complex of the circle becomes the join of the prior homotopy type with the parameter space $\tfrac{S^1}{\Z/(2k+1)}$ of new critical simplices.} 
\[\vr{S^1}{r}\simeq\underbrace{S^1 * \ldots * S^1}_{k+1} = S^{2k+1}.\]

\begin{figure}[h]
\includegraphics[width=\textwidth]{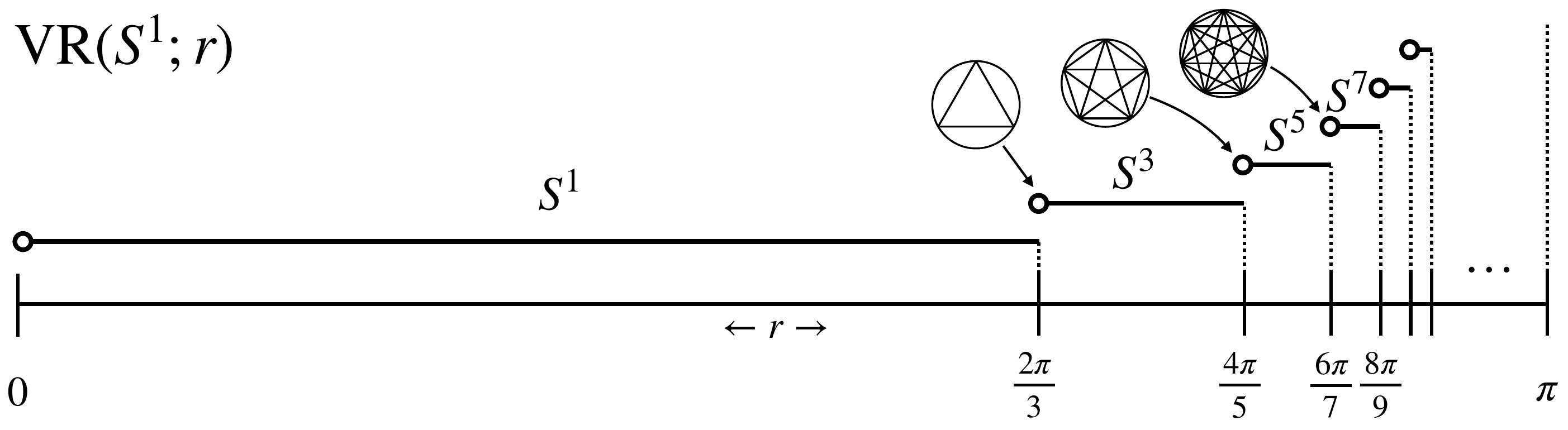}
\vspace{-2mm}
\caption{The homotopy types of $\vr{S^1}{r}$ as $r$ varies from $0$ to $\pi$~\cite{AA-VRS1}.
}
\label{fig:circle-barcodes}
\end{figure}

\begin{figure}[h]
\includegraphics[height=1in]{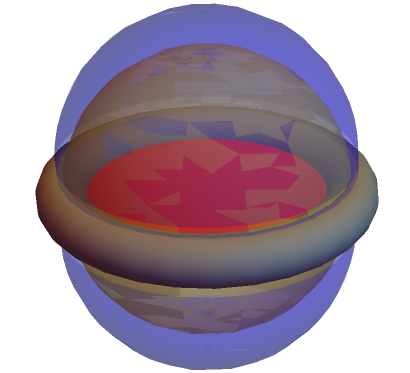}
\hspace{5mm}
\includegraphics[height=0.74in]{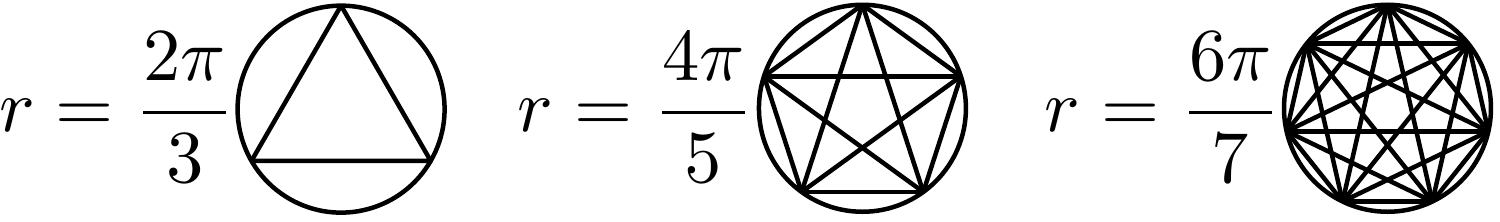}
\caption{
\emph{(Left)} A cartoon of the genus one Heegaard decomposition of the 3-sphere $S^3 = S^1 * S^1$ as the union of two solid tori.
\emph{(Right)} The critical simplices in the case of the circle.
}
\label{fig:circle-critical}
\end{figure}

Though this is the correct intuition for the homotopy types of the Vietoris--Rips complexes of the circle, this is not how the homotopy types were originally proven.
Instead,~\cite{AA-VRS1} proceeded by noting that $\vr{S^1}{r}$ is the union $\cup_{X\subseteq S^1\text{ finite}}\vr{X}{r}$ as $X$ varies over all finite subsets of $S^1$, by understanding Vietoris--Rips complexes of finite subsets of the circle and the maps between them, and then by using colimits and homotopy colimits.
Indeed, it is shown that if $\frac{2\pi k}{2k+1} < r \le \frac{2\pi (k+1)}{2k+3}$, then $\vr{X}{r}\simeq S^{2k+1}$ for all $X\subseteq S^1$ that are sufficiently close to $S^1$ in the Hausdorff distance.
The properties of homotopy colimits~\cite{Hocolim} then imply that $\vr{S^1}{r}$ is also homotopy equivalent to $S^{2k+1}$.
This proof builds upon Adamaszek's paper~\cite{Adamaszek2013} describing the Vietoris--Rips complexes of $n$ evenly-spaced points on the circle and the paper~\cite{AAFPP-J} describing the Vietoris--Rips and \v{C}ech complexes of arbitrary finite subsets of the circle.
See also~\cite{AAR,AAM,lim2023strange}.

Why doesn't the proof of the homotopy types of Vietoris--Rips complexes of the circle in~\cite{AA-VRS1} easily generalize to describe the homotopy types of Vietoris--Rips complexes $\vr{S^n}{r}$ of $n$-spheres for $n\ge 2$?
The proof of~\cite{AA-VRS1} relied on a near-complete understanding of the Vietoris--Rips complexes of finite subsets of the circle.
No such understanding is available for $n\ge 2$, although see~\cite{saleh2024vietoris} for a study of Vietoris--Rips complexes of platonic solids.

\subsection{One-dimensional persistence of geodesic spaces}
\label{ssec:S1dim}
Using knowledge of the Vietoris--Rips complexes of the circle,~\cite{gasparovic2018complete,virk2017approximations,virk20201,virk2022contractions} characterize properties of geodesic loops inside a manifold that are measured by the 1-dimensional persistent homology of the manifold's Vietoris--Rips complexes.
This improves the interpretability of persistent homology by identifying specific geometric features of a manifold that are visible in the persistent homology barcodes.
As a particular example, Corollary~5.8 of~\cite{virk2022contractions} shows that the persistent homology of $\vr{S^1_t}{r}$ appears in the Vietoris--Rips persistent homology of \emph{any} metric graph that is not a forest, where $t$ is the radius of the circle, and where $2\pi t$ is the circumference of the shortest loop in the graph.

Similarly, the \emph{systole} of a manifold (the length of the shortest noncontractible loop) can  be recovered as the smallest non-zero death in the manifold's 1-dimensional Vietoris--Rips persistent homology if the corresponding loop is homologically non-trivial~\cite{virk20201}.
More generally, the systole can always be recovered  as the smallest non-zero death in the manifold's 1-dimensional Vietoris--Rips persistent fundamental group~\cite{virk20201}.
Expanding on these results, if a loop corresponding to the systole is  homologically non-trivial, then under some mild geometric conditions the persistent homology of $\vr{S^1_t}{r}$ (rescaled by the systole) appears in the Vietoris--Rips persistent homology of the manifold~\cite{virk2021footprints}.
If a loop corresponding to the systole is nullhomologous, then similarly mild geometric conditions imply the systole appears as the birth of a two-dimensional bar~\cite{virkSelective}.

\subsection{Hausmann's question}
In~\cite[Problem~3.12]{Hausmann1995}, Hausmann asks if the the connectivity of $\vr{M}{r}$ grows larger until the scale $r$ exceeds the diameter of $M$, when $\vr{M}{r}$ becomes contractible.
This is disproven in~\cite{virk2021counter}, where Virk produces a compact Riemannian manifold $M$ such that $\conn(\vr{M}{r}) > \conn(\vr{M}{r'})$ even though $r<r'$.

Hausmann's theorem that $\vr{M}{r}$ is homotopy equivalent to the manifold $M$ for $r$ sufficiently small compared to the curvature of $M$ is foundational for the subject.
It has been extended in numerous ways, including homology reconstruction, extensions to non-manifold spaces, extensions to subsets of Euclidean space with positive reach, etc.; we refer the reader to~\cite{attali2022optimal,bauer2021gromov,ChazalOudot2008,komendarczyk2024topological,Latschev2001,majhi2022vietoris,majhi2023demystifying,niyogi2008finding,virk2021rips,virk2021counter}.
However, the original proof of Hausmann's theorem is difficult for the following reason: the inclusion $M \hookrightarrow \vr{M}{r}$ mapping each point in the manifold to a vertex in the Vietoris--Rips complex is not continuous, since the vertex set of a simplicial complex is equipped with the discrete topology whereas manifolds are not.
Hausmann produces the homotopy equivalence $\vr{M}{r}\simeq M$ by considering a map $T\colon \vr{M}{r}\to M$ which is not canonically defined, and indeed which depends on choosing a total order of all the points in the manifold $M$ (we contend this is not a natural thing to do).
Nevertheless, through techniques related to Whitehead's theorem, Hausmann proves that this map $T\colon \vr{M}{r}\to M$ is a homotopy equivalence, even without producing a homotopy inverse $M \to \vr{M}{r}$.
We recommend the reader also see Virk's natural proof using the nerve lemma~\cite{virk2021rips}.

In addition to the fact that the inclusion $M \hookrightarrow \vr{M}{r}$ need not be continuous, another indication that the topology on $\vr{M}{r}$ has some disadvantages is that it is not metrizable in general.
Indeed, a simplicial complex is metrizable if and only if it is locally finite, which means that each vertex is contained in only a finite number of simplices~\cite[Proposition~4.2.16(2)]{sakai2013geometric}.
For $M$ a manifold of dimension $n\ge 1$ and for $r>0$, the Vietoris--Rips complex $\vr{M}{r}$ is not metrizable.
This is despite the fact that the input $M$ is a metric space.
So Vietoris--Rips complexes (by design) change categories, from the category of metric spaces to the category of simplicial complexes.

\subsection{Vietoris--Rips metric thickenings}
Let $X$ be a metric space.
In~\cite{AAF}, Adamaszek, Adams, and Frick equip the set $|\vr{X}{r}|$ (the geometric realization of the simplicial complex) with a different topology---in fact a metric.
This new space $\vrm{X}{r}$ is called the \emph{Vietoris--Rips metric thickening}, and the superscrpt `$\mathrm{m}$' stands for `metric'.
Each point $\sum_{i=0}^k \lambda_i x_i \in |\vr{X}{r}|$ in the geometric realization (with the barycentric coordinates $\lambda_i>0$ satisfying $\sum_i \lambda_i=1$ and with $[x_0,\ldots,x_k]\in\vr{X}{r}$)  is now instead thought of as a probability measure $\sum \lambda_i \delta_{x_i}$, where each $\delta_{x_i}$ is a Dirac delta mass.
The metric thickening $\vrm{X}{r}$ is equipped with a $q$-Wasserstein (or Monge--Kantarovich--Rubenstein~\cite{vershik2013long,villani2003topics,villani2008optimal}) distance for $1\le q<\infty$.
The inclusion $X \hookrightarrow \vrm{X}{r}$ is now continuous in general.
Furthermore, $X \hookrightarrow \vrm{X}{r}$ is an isometry onto its image, and the Hausdorff distance between $X$ and $\vrm{X}{r}$ is at most $r$, meaning that $\vrm{X}{r}$ is an \emph{$r$-thickening of $X$} in the sense of Gromov~\cite{Gromov}.
Finally,~\cite[Theorem~4.2]{AAF} gives a version of Hausmann's theorem with a natural proof: the map $g\colon \vrm{M}{r}\to M$ from the Vietoris--Rips thickening of a manifold to the manifold, defined by mapping a probability measure to its Frech\'{e}t mean for $r$ sufficiently small compared to the curvature of $M$, is a homotopy equivalence, with the (now continuous) inclusion $M \hookrightarrow \vrm{M}{r}$ as its homotopy inverse.

The homotopy types of Vietoris--Rips simplicial complexes and Vietoris--Rips metric thickenings are known to be tightly connected.
This was proven in a series of recent papers producing stronger and stronger results.
If $X$ is finite then $\vr{X}{r}$ and $\vrm{X}{r}$ are homeomorphic (\cite[Proposition~6.2]{AAF}; see also~\cite[Equation~1.B(c)]{Gromov}).
Adams, Moy, M\'{e}moli, and Wang proved that for $X$ totally bounded, the persistent homology modules of $\vr{X}{r}$ and $\vrm{X}{r}$ are $\varepsilon$-interleaved for any $\varepsilon>0$~\cite{AMMW,MoyMasters}, and hence they have the same persistence diagrams.
We now know that the persistence modules for $\vr{X}{r}$ and $\vrm{X}{r}$ are furthermore isomorphic (we emphasize this is with the $<$ convention).
Indeed, Adams, Virk, and Frick proved that if $X$ is totally bounded, then $\vr{X}{r}$ and $\vrm{X}{r}$ have isomorphic homotopy groups~\cite{HA-FF-ZV}, but they do not show that they are weakly homotopy equivalent (they do not find a single map inducing these isomorphism).
Gillespie proved that they have isomorphic homology groups~\cite{gillespie2022homological}.
This sequence of results was notably strengthened in~\cite{gillespie2023vietoris}, when Gillespie proved that the natural continuous map $|\vr{X}{r}|\to\vrm{X}{r}$ defined by $\sum_i \lambda_i x_i \mapsto \sum_i \lambda_i \delta_{x_i}$ induces an isomorphism on all homotopy groups, and hence is a weak homotopy equivalence.
This implies that the persistent homology modules of $\vr{X}{r}$ and $\vrm{X}{r}$ are not only $\varepsilon$-interleaved for any $\varepsilon>0$, but furthermore isomorphic on the nose.

\subsection{Vietoris--Rips metric thickenings of spheres}
\label{ss:vrmSn}
What are the Vietoris--Rips metric thickenings of spheres?
The paper~\cite{AAF} identifies the first new homotopy type of Vietoris--Rips metric thickenings of the $n$-sphere for all $n$.
Indeed, let $r_n\coloneqq \arccos(\frac{-1}{n+1})$ be the diameter of the vertex set of a regular $(n+1)$-simplex inscribed in $S^n$.
We have $r_1=\frac{2\pi}{3}$ as shown in Figure~\ref{fig:circle-critical}, and Figure~\ref{fig:sphere-critical}(right) shows two inscribed regular tetrahedra in $S^2$ that each have diameter $r_2$.
Proposition~5.3 and Theorem~5.4 of~\cite{AAF} show that $\vrm{S^n}{r}$ is homotopy equivalent to $S^n$ for all $r<r_n$, and furthermore identifies the first new homotopy type $\vrmleq{S^n}{r_n} \simeq S^n * \frac{\so(n+1)}{A_{n+2}}$.
Here, the special orthogonal group $\so(n+1)$ is the set of all rotations of $\R^{n+1}$, and the alternating group $A_{n+2}$ encodes the rotational symmetries of a regular $(n+1)$-simplex inscribed in $S^n$.
So, the quotient group $\frac{\so(n+1)}{A_{n+2}}$ is the parameter space of all critical simplices (inscribed regular $(n+1)$-simplices) at scale $r_n$, and joining this parameter space with the prior homotopy type $S^n$ gives the first new homotopy type $\vrmleq{S^n}{r_n} \simeq S^n * \frac{\so(n+1)}{A_{n+2}}$.
We note that since $S^n$ is $(n-1)$-connected and $\frac{\so(n+1)}{A_{n+2}}$ is 0-connected, their join $S^n * \frac{\so(n+1)}{A_{n+2}}$ is $(n+1)$-connected
~\cite[Proposition~4.4.3]{matousek2003using}.
But what are the homotopy types of Vietoris--Rips metric thickenings of spheres at scales $r>r_n$?

\begin{figure}[htb]
\includegraphics[height=1in]{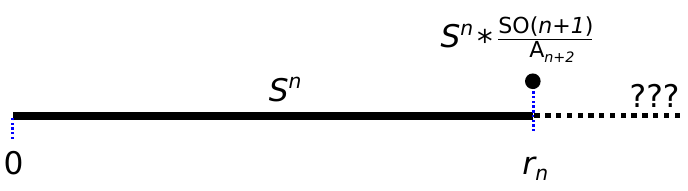}
\hspace{5mm}
\includegraphics[height=1.1in]{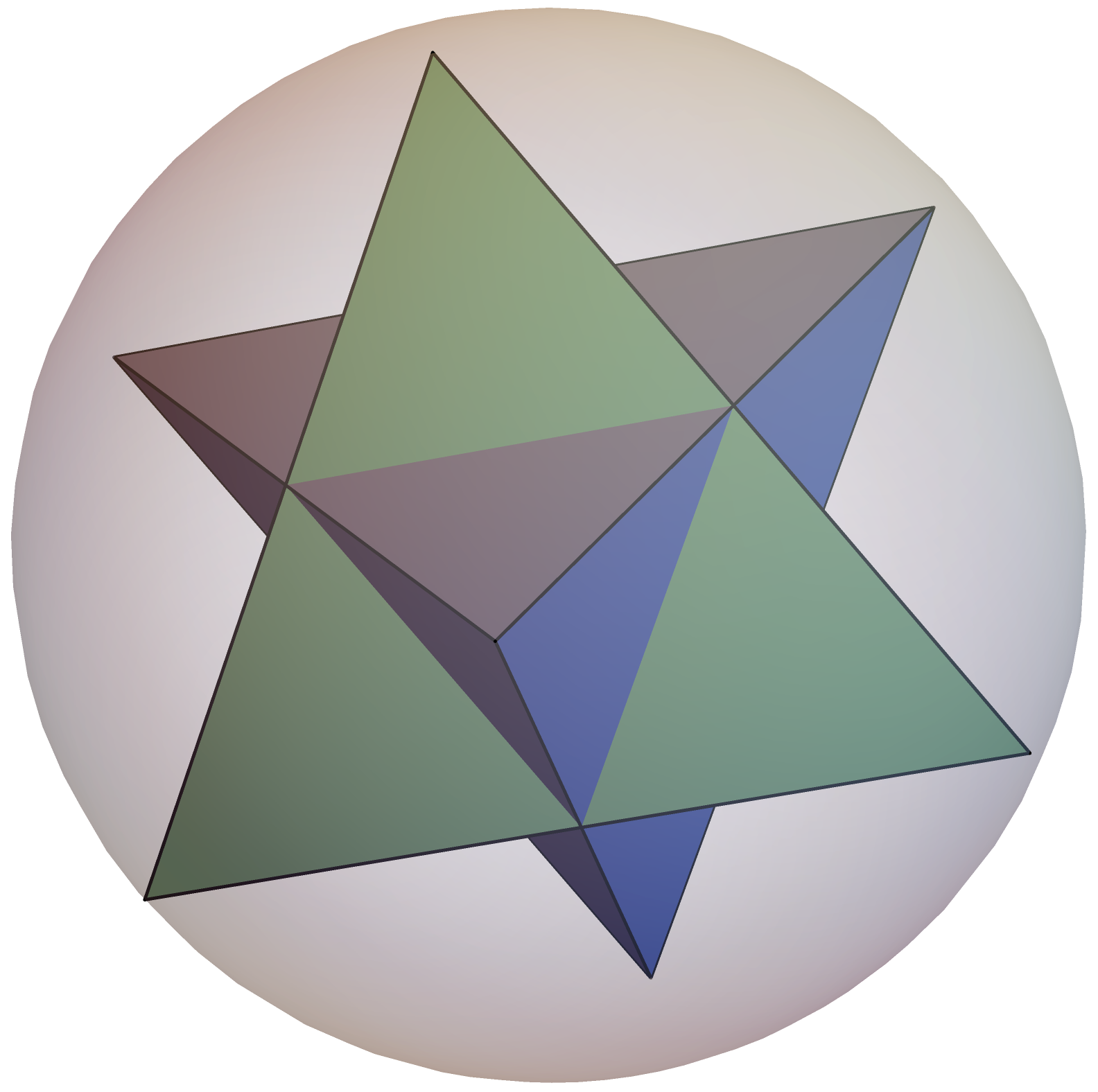}
\hspace{5mm}
\vspace{-2mm}
\caption{
\emph{(Left)}
The homotopy types of the Vietoris--Rips metric thickening $\vrm{S^n}{r}$ as $r$ increases~\cite{AAF}.
\emph{(Right)}
Two critical tetrahedra in $S^2$.}
\label{fig:sphere-critical}
\end{figure}

Another significant result is Moy's proof of the homotopy types of the Vietoris--Rips metric thickenings of the circle: for $\frac{2\pi k}{2k+1} \le r < \frac{2\pi (k+1)}{2k+3}$, we have a homotopy equivalence $\vrmleq{S^1}{r}\simeq S^{2k+1}$~\cite{moyVRmS1}.
Moy shows how to collapse the infinite-dimensional space $\vrmleq{S^1}{r}$ down to a finite CW complex.
This proof is more geometric than the one in~\cite{AA-VRS1}, and many of Moy's tools in~\cite{moyVRmS1} (including \emph{support homotopies}, for example) may be useful towards understanding the homotopy types of Vietoris--Rimps metric thickenings of spheres.
For these reasons, we view Moy's geometric proof in~\cite{moyVRmS1} as an improvement on the combinatorial proof in~\cite{AA-VRS1}.
Further improvements seem possible.
For example in the forthcoming work~\cite{PersistentEquivariantCohomology}, the authors construct an $S^1$-equivariant homotopy equivalence $\vrmleq{S^1}{r}\simeq S^{2k+1}$ for $\frac{2\pi k}{2k+1} \le r < \frac{2\pi (k+1)}{2k+3}$ (which is needed in order to study the persistent equivariant cohomology of the $S^1$ action on $\vrmleq{S^1}{r}$).

Realizing that determining exact homotopy types is a difficult problem, in~\cite{ABF2}, Adams, Bush, and Frick provide upper and lower bounds on the $\Z/2$ equivariant topology of $\vrm{S^n}{r}$ in terms of the covering and packing numbers of projective space.
In~\cite[Theorem~3]{ABF2} they show that if $\delta\ge 2\cdot \cov_{\RP^n}(k)$, then there is a continuous odd map $\vrm{S^n}{\pi-\delta}\to S^{k-1}$.
Therefore, by composition and the Borsuk--Ulam theorem~\cite{Matouvsek2008}, there cannot exist a continuous odd map $S^k \to \vrm{S^n}{\pi-\delta}$.
(In other words, the Vietoris--Rips metric thickening satisfies $\coind(\vrm{S^n}{\pi-\delta}) \le k-1$, where $\coind$ stands for the \emph{$\Z/2$-coindex}.)
On the other hand~\cite[Theorem~2]{ABF2}, shows that if $\delta\le \pack_{\RP^n}(k+1)$, meaning there are $k+1$ points in $\RP^n$ at pairwise distance at least $\delta$ apart, then there is a continuous odd map $S^k\to\vrm{S^n}{\pi-\delta}$ (meaning $k\le \coind(\vrm{S^n}{\pi-\delta})$).
Analogous bounds also apply to Vietoris--Rips simplicial complexes, and they were used in~\cite{GH-BU-VR} in order to provide new lower bounds on Gromov--Hausdorff distances between spheres.

\subsection{Borsuk--Ulam theorems}
Vietoris--Rips complexes of spheres have been used to provide generalizations of the Borsuk--Ulam theorems for maps $S^n \to \R^k$ with $k>n$; see~\cite{ABF,ABF2,BushMasters,BushThesis,crabb2023borsuk} and~\cite{akopyan2012borsuk,malyutin2023neighboring} for related results.
An improved understanding of the homotopy types of Vietoris--Rips complexes of spheres would enable improved Borsuk--Ulam theorems.

\subsection{Bestvina--Brady Morse theory}
In~\cite{zaremsky2022bestvina}, Zaremsky shows how to apply Bestvina--Brady Morse theory to Vietoris--Rips complexes of spheres at small scale parameters.
Bestvina--Brady Morse theory~\cite{bestvina1997morse} is a strong version of discrete Morse theory that predates Forman's discrete Morse theory~\cite{forman2002user} and that applies to infinite-dimensional simplicial complexes such as $\vr{S^n}{r}$.


\subsection{The Kuratowski embedding and the filling radius}

The \emph{filling radius} of a manifold $M$ was used by Gromov to prove the systolic inequality, which provides a lower bound for the volume of $M$ in terms of the systole (the length of the shortest non-contractible loop), and which enables a proof of the ``waist of the sphere" theorem~\cite{gromov1983filling,gromov2007metric}.
In subsequent work, Katz determined the filling radius of spheres and projective spaces~\cite{katz1983filling,katz1989diameter,katz9filling,katz1991neighborhoods}.
Work by Lim, M\'{e}moli, and Okutan~\cite{lim2022vietoris,okutan2019persistence} connects Vietoris--Rips complexes to quantitative topology: if $M$ is a manifold, then the top-dimensional bar in the Vietoris--Rips persistent homology dies at the filling radius.
On a similar note, recall from Section~\ref{ssec:S1dim} that the Vietoris--Rips persistent homology often captures the systole.

\subsection{The need for a Morse theory and a Morse--Bott theory}

There is a need to develop a stronger Morse theory~\cite{milnor1963morse} to describe the homotopy types of Vietoris--Rips complexes or thickenings of manifolds.
Critical points will correspond to configurations which can't be perturbed in order to reduce their diameter.
For symmetric spaces such as spheres where the critical simplices are not isolated, a Morse--Bott theory may be needed; Morse--Bott theory is a generalization of Morse theory that allows one to have entire submanifolds worth of critical points, instead of requiring that each critical point is isolated~\cite{bott1954nondegenerate,bott1982lectures,bott1988morse}.

There are at least three different settings in which such Morse theories could be developed:
\begin{itemize}
\item[(i)] Vietoris--Rips simplicial complexes $\vr{X}{r}$.
These have the disadvantage that $X\hookrightarrow \vr{X}{r}$ is not continuous if $X$ is not discrete, since the vertex set of a simplicial complex is equipped with the discrete topology.

\item[(ii)] Neighborhoods of the Kuratowski embedding, which are homotopy equivalent to Vietoris--Rips complexes by~\cite{lim2022vietoris,okutan2019persistence}.

\item[(iii)] Vietoris--Rips metric thickenings $\vrm{X}{r}$, as introduced in~\cite{AAF}.
The natural inclusion $X\hookrightarrow\vrm{X}{r}$ is now continuous (an isometry onto its image).
\end{itemize}

For (i), a Bestvina--Brady Morse theory for Vietoris--Rips complexes~\cite{zaremsky2022bestvina} has so far only been able to show $\vr{S^n}{r}\simeq S^n$ for $r$ sufficiently small, but not yet to identify any changes in homotopy type as $r$ increases.
See~\cite{gorivcan2023critical} for an analysis of critical edges in Vietoris--Rips complexes.
For (ii), three decades ago Katz developed a Morse theory for neighborhoods of the Kuratowski embedding~\cite{katz1983filling,katz1989diameter,katz1991neighborhoods}.
So far, this Morse theory is only strong enough to determine one new homotopy type of $\vr{S^1}{r}$ (revealing $S^3$ but none of $S^5, S^7, S^9, \ldots$), one new homotopy type of $\vr{S^2}{r}$, and no new homotopy types of $\vr{S^n}{r}$ for $n>2$.
For (iii), by contrast, the theory of metric thickenings is powerful enough to identify the first new homotopy type $\vrmleq{S^n}{r_n}\simeq S^n * \tfrac{\so(n+1)}{A_{n+2}}$ that occurs for all $n$, where $r_n$ is the diameter of a regular $(n+1)$-simplex inscribed in $S^n$.
A Morse theory for metric thickenings might build upon~\cite{AMMW,MirthThesis}.

Improved Morse or Morse-Bott theories for the settings (i), (ii), or (iii) would provide general tools for determining homotopy types of Vietoris--Rips complexes of manifolds, instead of having to invent new tools for each new space (which is currently the state of the art).
A satisfactory Morse theory ought to have both a first Morse lemma describing how $\vr{M}{r}\hookrightarrow \vr{M}{r'}$ or $\vrm{M}{r}\hookrightarrow \vrm{M}{r'}$ is a homotopy equivalence if there are no critical scales in $[r,r']$ as well as a second Morse lemma describing how the homotopy type changes upon passing a critical scale.
A satisfactory Morse theory ought to be able to handle the following examples.

\begin{example}[The circle $S^1$]
The first Morse lemma will say $\vrmleq{S^1}{r}\simeq S^1$ for $0\le r<\frac{2\pi}{3}$.
The second Morse lemma will say $\vrmleq{S^1}{\frac{2\pi}{3}}$ is obtained from $\vrm{S^1}{\frac{2\pi}{3}-\varepsilon}\simeq S^1$ by gluing on a circle's worth of 2-dimensional balls (simplices) $D^2$ along their boundaries via an attaching map $h$, yielding $\vrmleq{S^1}{\frac{2\pi}{3}}\simeq S^1\cup_h(D^2 \times S^1)\simeq S^1*S^1 = S^3$; see Figure~\ref{fig:circle-critical}(left).
Inductively, the first Morse lemma will say $\vrmleq{S^1}{r}\simeq S^{2k-1}$ for $\frac{2\pi (k-1)}{2k-1}<r<\frac{2\pi k}{2k+1}$.
Inductively, the second Morse lemma will say $\vrmleq{S^1}{\frac{2\pi k}{2k+1}}$ is obtained from $\vrm{S^1}{\frac{2\pi k}{2k+1}-\varepsilon}\simeq S^{2k-1}$ by gluing on a circle's worth of $2k$-dimensional balls (simplices) $D^{2k}$ along their boundaries via an attaching map $h$, yielding $\vrmleq{S^1}{\frac{2\pi k}{2k+1}}\simeq S^{2k-1}\cup_h(D^{2k} \times S^1)\simeq S^{2k-1}*S^1 = S^{2k+1}$~\cite{AA-VRS1,moyVRmS1}.
\end{example}

\begin{example}[The $n$-sphere $S^n$]
Let $r_n=\arccos(\frac{-1}{n+1})$ be the diameter of an inscribed regular $(n+1)$-dimensional simplex in $S^n$.
The first Morse lemma will say $\vrmleq{S^n}{r}\simeq S^n$ for $0\le r<r_n$.
The second Morse lemma will say $\vrmleq{S^n}{r_n}$ is obtained from $\vrm{S^1}{r_n-\varepsilon}\simeq S^n$ by gluing on a $\frac{\so(n+1)}{A_{n+2}}$ parameter space's worth of $(n+1)$-dimensional balls (simplices) $D^{n+1}$ along their boundaries via an attaching map $h$, yielding $\vrmleq{S^n}{r_n}\simeq S^n\cup_h(D^{n+1} \times \frac{\so(n+1)}{A_{n+2}})\simeq S^n * \frac{\so(n+1)}{A_{n+2}}$~\cite{AAF}; see Figure~\ref{fig:sphere-critical}.
\end{example}

\begin{example}
[Ellipses of small eccentricity]
\label{ex:ellipse}
An ellipse of course breaks the circular symmetry.
As we explain, the homotopy type of the Vietoris--Rips complexes of an ellipse begins at $S^1$, becomes $S^2$ once the first two inscribed equilateral triangles appear~\cite{AAR}, and only becomes $S^3$ once the entire circle's worth of inscribed equilateral triangles appear.

\begin{figure}[h]
\includegraphics[height=1.1in]{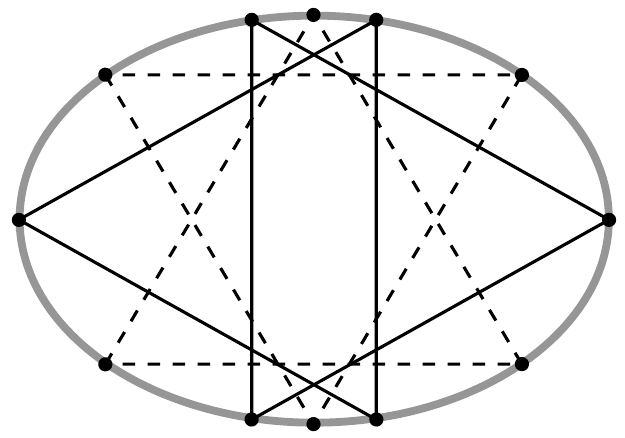}
\caption{The inscribed equilateral triangles with solid lines have diameter $r'$, and the inscribed equilateral triangles with dashed lines have diameter $r''$.}
\label{fig:ellipse}
\end{figure}

In more detail, let $E$ be an ellipse $\{(x,y)\in \R^2~|~(x/a)^2+y^2=1\}$ with $1<a\le\sqrt{2}$, with the Euclidean metric.
Let $r'=\frac{4\sqrt{3}a}{a^2+3}$ and $r''=\frac{4\sqrt{3}a^2}{3a^2+1}$ be the side-lengths of the smallest and largest inscribed equilateral triangles; see Figure~\ref{fig:ellipse}.
The first Morse lemma will say $\vrmleq{E}{r}\simeq E=S^1$ for $0\le r<r'$.
The second Morse lemma will say $\vrmleq{E}{r'}$ is obtained from $\vrm{E}{r'-\varepsilon}\simeq S^1$ by gluing on two $2$-dimensional balls (simplices) $D^2$ along their boundaries via an attaching map $h$, yielding $\vrmleq{E}{r'}\simeq S^1\cup_h(D^2 \times S^0)\simeq S^2$.
The first Morse lemma will say $\vrmleq{E}{r}\simeq S^2$ for $r'\le r<r''$~\cite{AAR}.
The second Morse lemma will say $\vrmleq{E}{r''}$ is obtained from $\vrm{E}{r''-\varepsilon}\simeq S^2$ by gluing on two $3$-dimensional balls $D^3$ along their boundaries via an attaching map $h$, yielding $\vrmleq{E}{r''}\simeq S^2\cup_h(D^3 \times S^0)\simeq S^3$.
\end{example}

\section{Notation and conventions}
\label{sec:notation} 

For $(X,d)$ a metric space and $x\in X$, we define the open and closed balls of radius $r$ by $B(x;r)\coloneqq \{y\in X~|~ d(x,y)< r\}$ and $B[x;r]\coloneqq \{y\in X~|~ d(x,y)\leq r\}$, respectively.
When needed, the ambient space will be denoted in the subscript, for example as $B_X(x;r)$.
For $k\geq 1$, we define the \emph{$k$-th covering radius of $X$} to be
\[\cov_X(k)\coloneqq \inf \{r\geq 0~|~\exists x_1,\dots, x_k\in X \text{ with 
} \cup_{i=1}^k B[x_i;r]\} = X. \]
For $\delta >0$, we define the \emph{covering number} of $X$ to be
\[\numCover_X(\delta)\coloneqq \min \{n\geq 1~|~ \exists x_1,\dots, x_n\in X\text{ with } \cup_{i=1}^n B[x_i;r]\} = X.\]

\begin{lemma}
\label{lem:cov-inf}
For a compact metric space $(X,d)$ and $k\geq 1$, the infimum in the definition of the $k$-th covering radius $\cov_X(k)$ is attained: there exist $x_1,\dots, x_k\in X$ with $\cup_{i=1}^k B[x_i;\cov_X(k)] = X$. 
\end{lemma}

\begin{proof}
Let $f \colon X^k \to \R$ be the continuous function defined via $(x_1, \ldots, x_k) \mapsto \max_{x\in X} d(x,\{x_1, ..., x_k\})$.
The sublevelsets $f^{-1}((-\infty,r])$ are closed and nonempty for $r>\cov_X(k)$.
Since $X^k$ is compact, the intersection of these sublevelsets over all $r>\cov_X(k)$ must also be nonempty by the finite intersection property.
Any element of this intersection is a tuple $(x_1, \ldots, x_k)$ with $d(x,\{x_1, ..., x_k\}) \le \cov_X(k)$ for all $x \in X$, giving $\cup_{i=1}^k B[x_i;\cov_X(k)] = X$.

\end{proof}

For $X$ a metric space and $r>0$, the \emph{Vietoris--Rips simplicial complex $\vr{X}{r}$} has vertex set $X$, and a finite set $\sigma\subseteq X$ as a simplex if its diameter $\diam(\sigma)$ is less than $r$.
We often identify a simplicial complex with its geometric realization.

Throughout this paper we equip the $n$-sphere $S^n$ with the geodesic metric.
So $S^n$ has diameter $\pi$.
However, all of our results hold if one instead prefers to use the Euclidean metric on $S^n$, after a (monotonic) change in scale parameter.

\section{Proof of Theorem~\ref{thm:main}}
\label{sec:proof}


Let $n\ge 1$ and $\delta>0$.
To prove our main result, Theorem~\ref{thm:main}, we must show that if $\conn(\vr{S^n}{\pi-\delta})=k-1$, then
\[\cov_{S^n}(2k+2) \le \delta < 2\cdot \cov_{\RP^n}(k).\]
We split the proof into two claims, from which the proof quickly follows.
Claim~\ref{claim:lower} gives a lower bound on the connectivity of $\conn(\vr{S^n}{\pi-\delta})$, and Claim~\ref{claim:upper} gives an upper bound.

\begin{claim}
\label{claim:lower}
If $\delta < \cov_{S^n}(2k+2)$ then $\conn(\vr{S^n}{\pi-\delta})\ge k$.
\end{claim}

\begin{claim}
\label{claim:upper}
If $\delta \ge 2\cdot \cov_{\RP^n}(k)$ then $\conn(\vr{S^n}{\pi-\delta})< k-1$.
\end{claim}

\begin{proof}[Proof of Theorem~\ref{thm:main}]
We need to prove that if $\conn(\vr{S^n}{\pi-\delta})=k-1$, then
\[\cov_{S^n}(2k+2) \le \delta < 2\cdot \cov_{\RP^n}(k).\]
Since $\conn(\vr{S^n}{\pi-\delta})\le k-1$, by Claim~\ref{claim:lower} it must be that $\cov_{S^n}(2k+2) \le \delta$.
And since $\conn(\vr{S^n}{\pi-\delta})\ge k-1$, by Claim~\ref{claim:upper} it must be that $\delta < 2\cdot \cov_{\RP^n}(k)$.
\end{proof}

It remains to prove Claims~\ref{claim:lower} and~\ref{claim:upper}, which we do in Sections~\ref{ssec:lower} and~\ref{ssec:upper}.
That will complete the proof of Theorem~\ref{thm:main}.
Afterwards, in Section~\ref{ssec:corollary}, we prove Corollaries~\ref{cor:main1} and~\ref{cor:main2}.

\subsection{Lower bound on connectivity: Proving Claim~\ref{claim:lower}}
\label{ssec:lower}

Our original proof relied on Proposition~3.1 of Meshulam's 2001 paper~\cite{meshulam2001clique}.
However, we later developed a shorter proof based on the results in an appendix by Barmak in a 2023 paper by Farber~\cite{farber2023large}.
Theorem~3.1 in the 2009 paper by Kahle~\cite{Kahle2009} could also be used.
We refer the reader to~\cite{aharoni2000hall,barmak2023connectivity,bendersky2023connectivity,Bjorner1995,chudnovsky2000systems,meshulam2003domination} for related results.

Following Barmak's appendix~\cite{farber2023large}, a simplicial complex $K$ is \emph{$\ell$-conic} if each subcomplex with at most $\ell$ vertices is contained in a simplicial cone, or equivalently, in the closed star $st_K(v)$ of some vertex $v\in K$.
Barmak proves that the following result.

\begin{theorem}[Theorem~4 of~\cite{farber2023large}]
If a simplicial complex $K$ is $(2k+2)$-conic for $k\ge 0$, then it is $k$-connected.
\end{theorem}

\begin{corollary}
\label{cor:intersection}
Let $X$ be a metric space.
If any $2k+2$ open balls of radius $r$ in $X$ intersect, then $\conn(\vr{X}{r})\ge k$.
\end{corollary}

\begin{figure}[h]
\includegraphics[width=\textwidth]{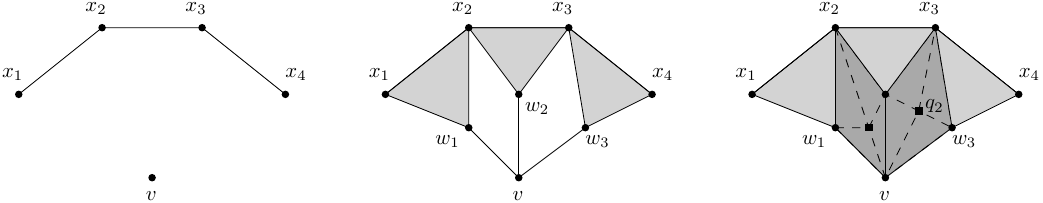}
\vspace{-2mm}
\caption{An example of Corollary~\ref{cor:intersection} when $k=1$.}
\label{fig:simplycon}
\end{figure}

Let $X$ be a metric space and let $r>0$.
Suppose that any four open $r$-balls in $X$ have a nonempty intersection, i.e., suppose that the hypothesis of Corollary~\ref{cor:intersection} is satisfied with $k=1$.
Corollary~\ref{cor:intersection} then claims that $\vr{X}{r}$ is simply-connected.
One way to see that this is true is sketched in Figure~\ref{fig:simplycon}.
Take a simplicial loop in $\vr{X}{r}$ determined by a sequence of vertices $x_0, x_1, x_2, \ldots, x_{k-1}, x_{k}=x_0$, and fix $v \in X$; see Figure~\ref{fig:simplycon}(left).
By assumption, for each index $i$ in $mod(k)$ arithmetic, there exists $w_i \in X$ at distance less than $r$ from $x_i, x_{i+1}$ and $v$, and hence $\vr{X}{r}$ contains each $2$-simplex spanned by $x_i, x_{i+1}, v$ and each $1$-simplex spanned by $w_i,v$; see Figure~\ref{fig:simplycon}(middle).
By assumption, for each $i$ in $mod(k)$ arithmetic, there exists $q_i \in X$ at distance less than $r$ from $w_i, x_{i+1}, w_{i+1},$ and $v$.
Hence for each $i$, $\vr{X}{r}$ contains the four $2$-simplices spanned by $w_i, x_{i+1}, q_{i}$, by $w_{i+1}, x_{i+1}, q_{i}$, by $w_{i+1}, v, q_{i}$, and by $w_{i}, v, q_{i}$.
These four $2$-simplices are sketched in Figure~\ref{fig:simplycon}(right), where one can observe that the $2$-simplices we have mentioned throughout this proof form a simplicial nullhomology of the original simplicial loop in $\vr{X}{r}$.

\begin{proof}[Proof of Corollary~\ref{cor:intersection}]
Suppose $L$ is a subcomplex of $\vr{X}{r}$ with at most $2k+2$ vertices $x_1,\ldots,x_\ell$, with $\ell\le 2k+2$.
By assumption, the intersection $\cap_{i=1}^\ell B_X(x_i;r)$ is nonempty. 
For any $v$ in this intersection, we have $d(v, x_i)<r$ for all $i$.
Hence, the subcomplex $L$ is contained in a simplicial cone in $\vr{X}{r}$ with apex $v$, or in other words, in the closed star $\st_{\vr{X}{r}}(v)$.
So $\vr{X}{r}$ is $(2k+2)$-conic, and it follows from~\cite[Appendix, Theorem~4]{farber2023large} that $\conn(\vr{X}{r})\ge k$.
\end{proof}

We are now prepared to prove Claim~\ref{claim:lower}.

\begin{proof}[Proof of Claim~\ref{claim:lower}]
Let $n\ge 1$ and $\delta>0$.
We must show that if $\delta < \cov_{S^n}(2k+2)$ then $\conn(\vr{S^n}{\pi-\delta})\ge k$.

Let $x_1,\ldots,x_{2k+2}\in S^n$.
Because no $2k+2$ closed balls of radius $\delta$ can cover $S^n$, it follows that the intersection $\cap_{i=1}^{2k+2} (S^n\setminus B_{S^n}[-x_i;\delta])=\cap_{i=1}^{2k+2} B_{S^n}(x_i;\pi-\delta)$ is nonempty. 
Hence any $2k+2$ open balls of radius $\pi-\delta$ in $S^n$ intersect, and so Corollary~\ref{cor:intersection} implies that $\conn(\vr{S^n}{\pi-\delta})\ge k$.
\end{proof}

Another example of Corollary~\ref{cor:intersection} is the following result.

\begin{corollary}
\label{cor:volume}
Let $X$ be a metric space equipped with a measure $\mu$.
If any open ball of radius $r$ in $X$ has measure greater than $\frac{2k+1}{2k+2}\mu(X)$, then $\conn(\vr{X}{r})\ge k$.
\end{corollary}

\begin{proof}
The intersection of any $2k+2$ open balls of radius $r$ in $X$ has positive measure, and hence is nonempty.
Therefore Corollary~\ref{cor:intersection} gives $\conn(\vr{X}{r})\ge k$.
\end{proof}

A convenient example is when $X$ is a homogeneous manifold in which all balls of the same radius have the same volume, such as tori $(S^1)^n$, projective spaces $\RP^n$, or complex projective spaces $\CP^n$.
Of course, the bound in Corollary~\ref{cor:volume} is not novel if some ball $B(x;r)$ is equal to all of $X$, since then $\vr{X}{r}$ is contractible since it is a cone.

\subsection{Upper bound on connectivity: Proving Claim~\ref{claim:upper}}
\label{ssec:upper}


Our proof will use spaces equipped with $\Z/2$ actions and odd (or $\Z/2$-equivariant) maps.
Let $X$ and $Y$ be two spaces equipped with $\Z/2$ actions, denoted by $x\mapsto -x$ and $y\mapsto -y$.
We say that a map $f\colon X\to Y$ is \emph{odd} if $f(-x)=-f(x)$ for all $x\in X$.
In the case of $S^n$, the $\Z/2$ action sends $x$ to its antipodal point $-x$.
The Borsuk--Ulam theorem~\cite{matousek2003using} states that there cannot exist a continuous $\Z/2$ map from $S^k$ to $S^{k-1}$.
This is true even in the case $k=0$, when $S^0$ consists of two points and $S^{-1}$ is the emptyset.

As explained in the proof of~\cite[Proposition~5.3.2(iv)]{matousek2003using}, if $X$ is a space with a $\Z/2$ action that is $(k-1)$-connected, then there exists an odd continuous map $S^k\to X$.
Indeed, for the base case $k=0$, if $\conn(X)\ge -1$, then $X$ is nonempty.
By picking an arbitrary point $x\in X$ we produce a map from a point into $X$, and then by reflecting this point under the $\Z/2$ action we obtain a $\Z/2$ map from $S^0$ into $X$.
If $X$ is 0-connected, then there is a path between $x$ and $-x$, and then by reflecting this path under the $\Z/2$ action we obtain a $\Z/2$ map from $S^1$ into $X$.
Proceeding inductively, we obtain a $\Z/2$ map from $S^{k-1}$ into $X$.
If $X$ is $(k-1)$-connected, this map extends via a nullhomotopy $D^k\to X$, and by reflecting this nullhomotopy under the $\Z/2$ action we obtain a $\Z/2$ map from $S^k$ into $X$.

Proposition~5.3.2(iv) of~\cite{matousek2003using} concludes that if $X$ is a space with a $\Z/2$ action that is $(k-1)$-connected, then there cannot exist an odd continuous map $X\to S^{k-1}$.
Otherwise, composing this map with the one $S^k\to X$ constructed in the paragraph above would produce an odd continuous map $S^k \to S^{k-1}$, contradicting the Borsuk--Ulam theorem.

In the case of a Vietoris--Rips simplicial complex of a sphere, the $\Z/2$ action on $\vr{S^n}{r}$ is induced from the $\Z/2$ action on $S^n$, meaning $-(\sum_j \lambda_j y_j)=\sum_j \lambda_j (-y_j)$ for every point $\sum_j \lambda_j y_j$ in the geometric realization of $\vr{S^n}{r}$.
We note that the $\Z/2$ action on $\vr{S^n}{r}$ is free if $r<\pi$.

\begin{proof}[Proof of Claim~\ref{claim:upper}]
Let $n\ge 1$ and $\delta>0$.
We must show that if $\delta \ge 2\cdot \cov_{\RP^n}(k)$, then $\conn(\vr{S^n}{\pi-\delta})<k-1$.

We mimic the construction of~\cite[Theorem~3]{ABF2}, using Vietoris--Rips simplicial complexes instead of Vietoris--Rips metric thickenings, in order to obtain a continuous odd map $\vr{S^n}{\pi-\delta}\to S^{k-1}$.
Suppose $\delta \geq 2\cdot \cov_{\RP^n}(k)$. 
By Lemma~\ref{lem:cov-inf}, the infimum in the definition of the covering radius is attained, and there exist $[x_1],\dots, [x_k]\in\RP^n$ such that every point in $\RP^n$ is within distance $\tfrac{\delta}{2}$ from one of the $[x_i]$.
After lifting each $[x_i]$ to a pair of antipodal points $x_i$ and $-x_i$ in $S^n$, suppose $y,y'\in S^n$ are such that $d_{S^n}(y,x_i)\leq \tfrac{\delta}{2}$ and $d_{S^n}(y',-x_i)\leq \tfrac{\delta}{2}$. 
It follows that $d_{S^n}(y,y')\geq \pi-\delta$.
Therefore, no subset of $S^n$ of diameter less than $\pi-\delta$ intersects both $B[x_i;\tfrac{\delta}{2}]$ and $B[-x_i;\tfrac{\delta}{2}]$.

\begin{figure}[htb]
\centering
\includegraphics[width=2in]{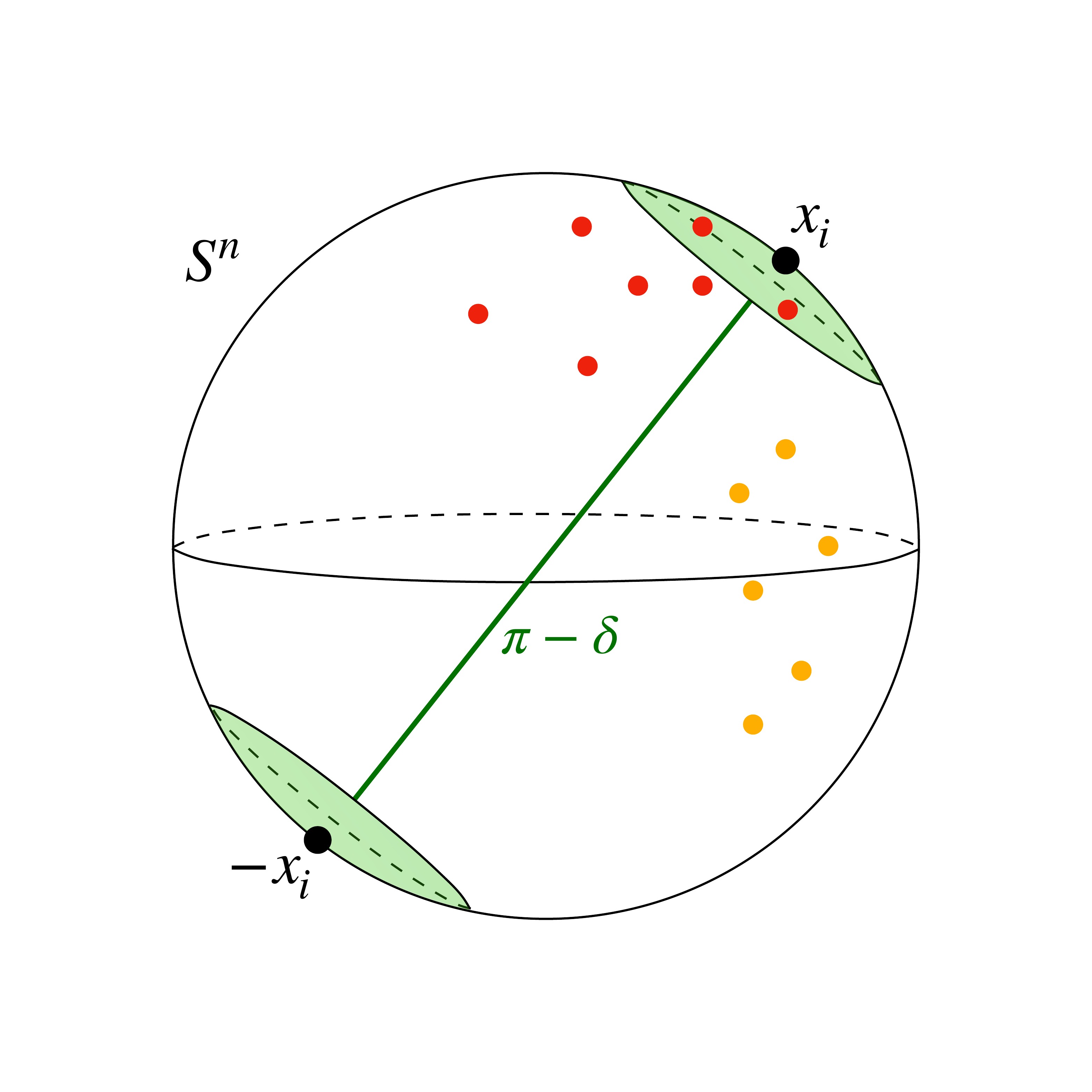}
\caption{An illustration of the odd map $f_i\colon \vr{S^n}{\pi-\delta}\to \R$.
A measure $\mu$ supported on the red points will have $f_i(\mu)>0$, whereas a measure $\nu$ supported on the orange points will have $f_i(\nu)=0$.}
\label{fig:proof-odd-map}
\end{figure}

For each $i=1,\dots, k$, define the odd map $f_i\colon \vr{S^n}{\pi-\delta}\to \R$ by the rule 
\[f_i(\textstyle{\sum}_{j}\lambda_j y_j) = \begin{cases}
\phantom{-}\sum_j \lambda_j d(y_j, S^n \setminus B[x_i;\tfrac{\delta}{2}] ) & \text{if }  \{y_j\}_j \cap B[-x_i;\tfrac{\delta}{2}]=\emptyset \\ 
-\sum_j \lambda_j d(y_j, S^n \setminus B[-x_i;\tfrac{\delta}{2}] ) & \text{if }  \{y_j\}_j \cap B[x_i;\tfrac{\delta}{2}]=\emptyset.  
\end{cases}\]
See Figure~\ref{fig:proof-odd-map}.
Observe that $f_i$ is well-defined: $\diam(\{y_j\}_j)<\pi-\delta$ implies at most one of the two cases above holds, and if both cases hold then they agree and give $f_i(\textstyle{\sum}_{j}\lambda_j y_j)=0$.
To see that $f_i$ is continuous on the geometric realization of $\vr{S^n}{\pi-\delta}$, note that $f_i$ is linear in each simplex $\{y_j\}_j$ (with respect to the barycentric coordinates $\{\lambda_j\}_j$), with agreement whenever two simplices share a face.
Since the balls $\big\{B\big[[x_i];\frac{\delta}{2}\big]\big\}_i$ cover $\RP^n$, the balls $\big\{B[x_i;\frac{\delta}{2}]\big\}_i \cup \big\{B[-x_i;\frac{\delta}{2}]\big\}_i$ cover $S^n$, and therefore the map 
$(f_1,\ldots,f_k)\colon \vr{S^n}{\pi-\delta} \to \R^k\setminus\{\vec{0}\}$ 
misses the origin.
By normalizing, we obtain an odd and continuous map
\[\tfrac{(f_1,\dots, f_k)}{\lVert (f_1,\dots, f_k) \rVert}\colon \vr{S^n}{\pi-\delta}\to S^{k-1}.\]
It follows from~\cite[Proposition~5.3.2(iv)]{matousek2003using} that $\vr{S^n}{\pi-\delta}$ is not $(k-1)$-connected.
\end{proof}

We remark that there is an analogue of~\cite[Proposition~5.3.2(iv)]{matousek2003using} for more general groups $G$.
Indeed,~\cite[Lemma~6.2.2]{matousek2003using} implies that if $X$ is a space with a $G$ action that is $(k-1)$-connected, then there exists a $G$-equivariant continuous map from the $k$-skeleton of $EG$ to $X$, where $EG$ is the the total space of the universal bundle over the classifying space $BG$~\cite{tom1987transgroups}.
In other words, the connectivity of $X$ means there are no obstructions to defining the desired $G$-equivariant map.

\subsection{Proof of corollaries}
\label{ssec:corollary}

We prove Corollaries~\ref{cor:main1} and~\ref{cor:main2}.

\begin{proof}[Proof of Corollary~\ref{cor:main1}]
Let $n\ge 1$ and $\delta>0$.
We must prove that 
\[\tfrac{1}{2}\numCover_{S^n}(\delta)-2 \le \conn(\vr{S^n}{\pi-\delta}) \le \numCover_{\RP^n}(\tfrac{\delta}{2})-2.\]

Let $\conn(\vr{S^n}{\pi-\delta}=k-1$.
By Theorem~\ref{thm:main} we have $\delta < 2\cdot \cov_{\RP^n}(k)$.
This means $k < \numCover_{\RP^n}(\tfrac{\delta}{2})$, or equivalently, $\conn(\vr{S^n}{\pi-\delta}) \le \numCover_{\RP^n}(\tfrac{\delta}{2})-2$.

Let $\conn(\vr{S^n}{\pi-\delta}=k-1$.
Also by Theorem~\ref{thm:main} we have $\cov_{S^n}(2k+2) \le \delta$.
Since the infimimum in the definition of the covering radius is realized (Lemma~\ref{lem:cov-inf}), this gives $\numCover_{S^n}(\delta)\le 2k+2$.
Equivalently, $\tfrac{1}{2}\numCover_{S^n}(\delta) \le k+1 = \conn(\vr{S^n}{\pi-\delta}+2$.
\end{proof}

\begin{proof}[Proof of Corollary~\ref{cor:main2}]
We must prove that for $n\ge 1$ and $\varepsilon>0$, the homotopy type of $\vr{S^n}{r}$ changes infinitely many times as $r$ varies over the interval $(\pi-\varepsilon,\pi)$.

By Corollary~\ref{cor:main1} we have
\[\tfrac{1}{2}\numCover_{S^n}(\delta)-2 \le \conn(\vr{S^n}{\pi-\delta}) \le \numCover_{\RP^n}(\tfrac{\delta}{2})-2.\]
The righthand inequality shows $\conn(\vr{S^n}{r})<\infty$ for all $r<\pi$, and the lefthand inequality shows that as $r$ tends to $\pi$ from below, $\conn(\vr{S^n}{r})$ tends to $\infty$.
Therefore $\conn(\vr{S^n}{r})$ attains infinitely many different values as $r$ varies over $(\pi-\varepsilon,\pi)$.
So $\vr{S^n}{r}$ attains infinitely many different homotopy types over this range, as well.
\end{proof}

\section{Example explicit connectivity bounds}
\label{sec:examples}

In this section, we give examples of explicit bounds on the connectivity of $\vr{S^n}{r}$ obtained from Theorem~\ref{thm:main} by inserting known bounds on the covering radii of spheres and of projective spaces for $n=1$ and $n=2$.

\subsection{The case $n=1$}
\label{ssec:examples-n1}

It is clear that the covering radii of the circle are realized by balls centered at evenly-spaced points; hence $\cov_{S^1}(2k+2)=\tfrac{\pi}{2k+2}$. 
Similarly, viewing $\RP^1$ as the geodesic circle of circumference $\pi$, it follows that $\cov_{\RP^1}(k)=\tfrac{\pi}{2k}$. 

On the other hand, it is known from~\cite{AA-VRS1} that if $\frac{\pi}{k+2} \le \delta < \frac{\pi}{k}$ with $k$ odd, then $\vr{S^1}{\pi-\delta}\simeq S^k$ is homotopy equivalent to an odd sphere of dimension $k$, and hence $\conn(\vr{S^1}{\pi-\delta})=k-1$.
So, in the case $n=1$, we have $\conn(\vr{S^1}{\pi-\delta})=k-1$ if and only if
\begin{itemize}
\item $k$ is odd and $\delta\in[\frac{\pi}{k+2},\frac{\pi}{k})$, which is a subset of the interval $[\cov_{S^1}(2k+2),2\cdot \cov_{\RP^1}(k))=[\frac{\pi}{2k+2},\frac{\pi}{k})$ given in Theorem~\ref{thm:main}. 
We note that the righthand endpoint $\frac{\pi}{k}$ obtained from Claim~\ref{claim:upper} in the case of $n=1$ and $k$ odd is tight, however.
\item $k$ is even and $\delta\in \emptyset$, which is again of course a subset of the interval given in Theorem~\ref{thm:main}.
\end{itemize}

\begin{figure}[h]
\includegraphics[width=0.95\linewidth]{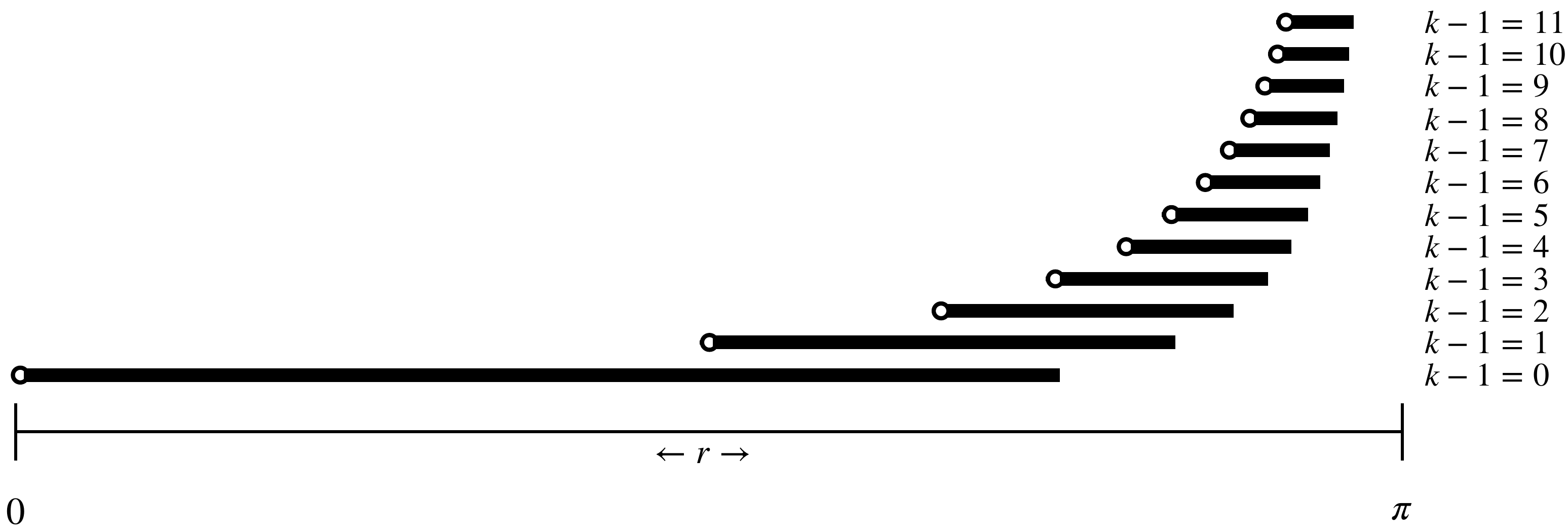}
\caption{
Intervals where $\vr{S^1}{r}$ may have connectivity $k-1$.
Cf.\ Figure~\ref{fig:s1_intervals-intro}. 
}
\label{fig:S1-intervals}
\end{figure}

\subsection{The case $n = 2$}
\label{ssec:examples-n2}

Figure~\ref{fig:S2-intervals} shows intervals where $\vr{S^2}{r}$ may have connectivity $(k-1)$ for small values of $k$.
The endpoints of these intervals are plotted using bounds for $\cov_{S^2}(2k+2)$ and $\cov_{\RP^2}(k)$ obtained by L.\ T{\' o}th~\cite{toth1943covering}, G.\ T{\' o}th~\cite{toth1969}, Sch{\"u}tte~\cite{schutte1955uberdeckungen}, Jucovi{\v c}~\cite{jucovic1960some}, Tarnai, G{\'a}sp{\'a}r, and Fowler~\cite{tarnai1985covering, tarnai1986covering, tarnai1991covering, fowler2002circle};  see Table~\ref{tab:bounds}.
Many of these values  are known to be tight (equalities, not just upper bounds).

\begin{figure}[h]
\includegraphics[width=0.95\linewidth]{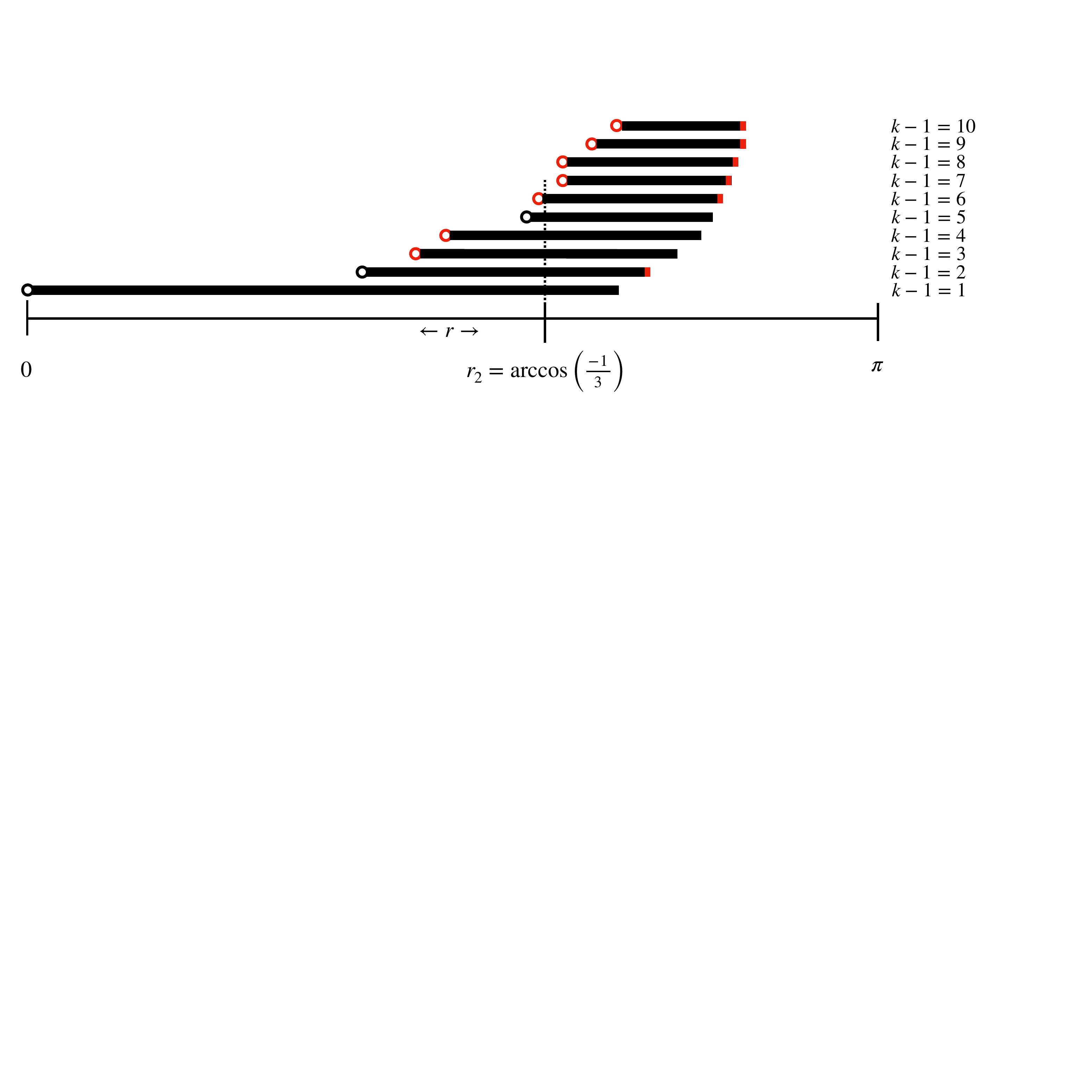}
\caption{
Intervals where $\vr{S^2}{r}$ may have connectivity $k-1$.
The red endpoints are plotted using only approximate values of $\cov_{S^2}(2k+2)$ or $\cov_{\RP^2}(k)$ that bound the true values; see Table~\ref{tab:bounds} for a subset of these values.
}
\label{fig:S2-intervals}
\end{figure}




\begin{table}[htbp]
\centering
\begin{tabular}{|c|c|c|}
\hline
\textbf{$k$} & Upper bound for $\cov_{S^2}(2k+2)$~\cite{tarnai1991covering} & Upper bound for $\cov_{\RP^2}(k)$~\cite{fowler2002circle} \\
\hline \hline
1 & \textcolor{blue}{$\arccos(\tfrac{1}{3}) \approx 1.230959$} 
& \textcolor{blue}{$\tfrac{\pi}{2}\approx 1.570796$} 
\\
\hline
2 & \textcolor{blue}{$\tfrac{1}{2}\arccos(\tfrac{-1}{3})\approx 0.955317$} 
& \textcolor{blue}{$\tfrac{\pi}{2}\approx 1.570796$} 
\\
\hline
3 & 0.840193  & \textcolor{blue}{$\tfrac{1}{2}\arccos(\tfrac{-1}{3})\approx 0.955317$} 
\\
\hline
4 & \textcolor{blue}{0.738411} 
& 0.857072\\
\hline
5 & \textcolor{blue}{$\arccos\left(\sqrt{\tfrac{5+2\sqrt{5}}{15}}\right)\approx 0.652358$} 
& 0.801530 \\
\hline
6 & \textcolor{blue}{0.609782} 
& \textcolor{blue}{$\arccos\left(\sqrt{\tfrac{5+2\sqrt{5}}{15}}\right) \approx 0.652358$} 
\\
\hline
7 & 0.574193 & 0.631914\\
\hline
\end{tabular}
\vspace{1em}
\caption{Bounds for $\cov_{S^2}(2k+2)$ and $\cov_{\RP^2}(k)$ in radians.
Bounds in blue are tight.
The values for $\cov_{S^2}(2k+2)$ correspond to the tetrahedron for $k=1$, the octahedron for $k=2$, and the icosahedron for $k=5$.
The values for $\cov_{\RP^2}(k)$ correspond to the octahedron for $k=3$ and the icosahedron for $k=6$.}
\label{tab:bounds}
\end{table}

\section{Rigidity of connectedness}
\label{sec:rigidity}

Let $M$ be a metric space and let $r>0$ be such that $\vr{M}{r}$ is $k$-connected.
There may exist metric spaces arbitrarily close to $M$ in the Gromov--Hausdorff metric whose Vietoris--Rips complexes at scale $r$ are not $k$-connected.
Indeed, the stability of persistent homology~\cite{ChazalDeSilvaOudot2014} only implies that $X$ and $M$ have similar persistent homology. 

In this section we observe that the lower bound for $k$-connectedness for the Vietoris-Rips complexes in Corollary~\ref{cor:intersection} is stable, in some sense.
Given a fixed $r$ for which we have proven the $k$-connectedness for the Vietoris-Rips complexes of $M$ using Corollary~\ref{cor:intersection}, the mentioned $k$-connectedness is rigid: the same connectedness bound also holds for the Vietoris--Rips complex built on a metric space $X$ that is sufficiently close to $M$ in the Gromov--Hausdorff metric.

\begin{theorem}
\label{thm:rigid1}
Let $M$ be a metric space, $k$ an integer, and $\rho>0$ such that each collection of $(2k+2)$ open balls of radius $\rho$ in $M$ has a nonempty intersection.
Then for each $r > \rho$, $\vr{M}{r}$ is $k$-connected.
Furthermore, this $k$-connectedness is rigid in the following sense: for each metric space $X$ satisfying $d_\gh(X,M) < \frac{r-\rho}{2}$, the complex $\vr{X}{r}$ is $k$-connected.
\end{theorem}

\begin{proof}
It suffices to prove the last conclusion of the theorem. 
Let $0 < \nu < \frac{r-\rho}{2}$.
We will show that if the metric space $X$ satisfies $d_\gh(X,M) < \nu$, then the complex $\vr{X}{r}$ is $k$-connected.

Let $d$ denote a metric on $X \coprod M$ in which the Hausdorff distance between the naturally embedded $X$ and $M$ is less than $\nu$.
Such a metric exists by the definition of the Gromov--Hausdorff distance.
For each $y\in M$ choose some $\tilde{y} \in X$ such that $d(y, \tilde{y})< \nu$.

With the aim of applying Corollary~\ref{cor:intersection}, choose any $x_1, x_2, \ldots, x_{2k+2} \in X$.
For each $i$ choose $x'_i \in M$ such that $d(x_i, x'_i) < \nu$.
The intersection
$
\cap_{i=1}^{2k+2} B_{M} (x'_i; r-2\nu)
$, 
is nonempty by assumption, since $r-2\nu>\rho$.
So there exists some $m\in M$ with $d(m, x'_i) < r - 2\nu$ for all $i$.
By the triangle inequality we deduce that for $\tilde{m} \in X$, we have $d(\tilde{m}, x_i) < r$ for all $i$.
So $\cap_{i=1}^{2k+1}B_X(x_i;r)\neq \emptyset$, and so Corollary~\ref{cor:intersection} implies $\conn(\vr{X}{r})\ge k$.
\end{proof}

For example, let $M=S^n$.
In the proof of Claim~\ref{claim:lower}, we proved that any collection of $(2k+2)$ open balls of radius greater than $\pi-\cov_{S^n}(2k+2)$ is nonempty.
So $\vr{S^n}{r}$ is $k$-connected for
\[
r \in \Big(\pi-\cov_{S^n}(2k+2), \infty\Big).
\]
This interval of $k$-connectivity is stable in the following sense: 
for any metric space $X$, we have that
$\vr{X}{r}$ is $k$-connected for 
\[r \in \Big(\pi-\cov_{S^n}(2k+2)+2d_\gh(X,S^n), \infty\Big).\]

\section{Conclusion}
\label{sec:conclusion}

We end with some conjectures and open questions that we hope will help inspire further research.

\begin{conjecture}
Though Hausmann's Problem~3.12~\cite{Hausmann1995} has a negative answer in general (see~\cite{virk2021counter}), we conjecture that it has an affirmative answer for Vietoris--Rips complexes of spheres.
Explicitly, if $S^n$ is the $n$-sphere, then we conjecture that $\conn(\vr{S^n}{r})$ is a non-decreasing function of the scale parameter $r$.
\end{conjecture}

\begin{conjecture}[Conjecture~5.7 from~\cite{AAF}]
We conjecture that for all $n$, there is some $\varepsilon>0$ such that $\vr{S^n}{r}\simeq S^n * \frac{\so(n+1)}{A_{n+2}}$ for all $r_n < r < r_n+\varepsilon$.
(See Section~\ref{ss:vrmSn} for the definition of $r_n$.)
\end{conjecture}

\begin{conjecture}
\label{conj:finite-cw}
We conjecture that $\vr{S^n}{r}$ is homotopy equivalent to a finite CW complex for all $r\ge 0$.
\end{conjecture}

These three conjectures are all true in the case $n=1$ by~\cite{AA-VRS1}.

Note that Conjecture~\ref{conj:finite-cw} is false if one uses the $\le$ convention (instead of the $<$ convention we are using in this paper), since $\vrleq{S^1}{\frac{2\pi k}{2k+1}}$ is homotopy equivalent to an uncountable wedge sum of $2k$-dimensional spheres for each integer $k\ge 0$~\cite{AA-VRS1}.

\begin{question}
Fix $n$ and $r$.
What is the smallest integer $k$ such that $\vr{S^n}{r}$ is homotopy equivalent to a CW complex of dimension $k$?

We remark that if a space $Y$ is homotopy equivalent to a CW complex of dimension $k$, then the $i$-dimensional homology $H_i(Y)$ and cohomology $H^i(Y)$ are zero for all $i\ge k+1$.
\end{question}

\begin{question}
Let $M$ be a compact Riemannian manifold.
What assumptions need to be placed on $M$ so that $\vrm{M}{r}$, $\vrmleq{M}{r}$, and $\vr{M}{r}$ always homotopy equivalent to a finite CW complex?




We note that the metric thickening $\vrmleq{S^1}{r}$ is always homtopy equivalent to a finite CW complex by~\cite{moyVRmS1}.
As mentioned above, the simplicial complex $\vrleq{S^1}{\frac{2\pi k}{2k+1}}$ is homotopy equivalent to an uncountable wedge sum of $2k$-dimensional spheres for each integer $k\ge 0$~\cite{AA-VRS1}.

What assumptions need to be placed on $M$ so that $\vrm{M}{r}$, $\vrmleq{M}{r}$, $\vr{M}{r}$, and $\vrleq{M}{r}$ are always homotopy equivalent to a finite-\emph{dimensional} CW complex?
\end{question}

\begin{question}
Recall Conjecture~\ref{conj:countable}.
Can we identify a set $C \subsetneq [0,\pi]$ of potential \emph{critical scale parameters} satisfying the ``first Morse lemma'' property that if $(s,s')\cap C=\emptyset$ for $s<s'$, then the inclusion $\vr{S^n}{r}\hookrightarrow \vr{S^n}{r'}$ is a homotopy equivalence for all $s<r\le r'<s'$?

In the case of the circle $n=1$, we have that $C=\{\frac{2\pi k}{2k+1}~|~k\in\mathbb{N}\}\cup\{\pi\}$ is the minimal set of critical scale parameters~\cite{AA-VRS1}.
We refer the reader to~\cite{lovasz1983self,katz1989diameter,MirthThesis,zaremsky2022bestvina,katz2023extremal,gorivcan2023critical} for related ideas.
\end{question}

\begin{question}
In particular, what is the relationship between the diameters of strongly self-dual polytopes in Lov{\'a}sz' paper~\cite{lovasz1983self} and the critical scale parameters of $\vr{S^n}{r}$?
We note that every vertex of such a polytope is paired with other vertices realizing the diameter~\cite[Lemma~2]{lovasz1983self}, which should be the case for critical simplices in $\vr{S^n}{r}$.
The strongly self-dual polytopes in the circle are precisely the regular odd polygons, corresponding to the critical scales of $\vr{S^1}{r}$.

What are all possible strongly self-dual polytopes in $S^n$ for $n\ge 2$?
See also~\cite{horvath2021strongly}.
Theorem~1 of~\cite{lovasz1983self} shows that for any $n\ge 1$, there are strongly self-dual polytopes in $S^n$ whose finite vertex sets are arbitrarily dense.
But, due to its inductive nature, Lov{\'a}sz' construction almost certainly does not produce all strongly self-dual polytopes.
\end{question}




\begin{question}
Proposition~5.3 of~\cite{AA-VRS1} implies that if $r$ is not a critical scale parameter and if the finite subset $X\subseteq S^1$ is sufficiently dense in $S^1$ (i.e.\ sufficiently close to $S^1$ in the Hausdorff distance), then the inclusion $\vr{X}{r}\hookrightarrow \vr{X'}{r}$ is a homotopy equivalence for all larger finite subsets $X \subseteq X' \subseteq S^1$.
This fact was used in the first proof of the homotopy types of $\vr{S^1}{r}$.
Is a similar fact true when $S^1$ is replaced with $S^n$ for $n\ge 2$?
\end{question}

\begin{question}
The papers~\cite{AAFPP-J,AA-VRS1} show that if $X\subseteq S^1$ is finite, then the only possible homotopy types of $\vr{X}{r}$ for $r>0$ are a single odd-dimensional sphere or a wedge sum of even-dimensional spheres of the same dimension (we consider the point to be a 0-fold wedge sum).
This fact was used in the proof of the homotopy types of $\vr{S^1}{r}$.
For $n\ge 2$ and $r > r_n = \arccos(\frac{-1}{n+1})$, is it possible to describe the homotopy types of $\vr{X}{r}$ for $X \subseteq S^n$ finite?

We restrict attention to $r>r_n$ for the following reason.
Any simplicial complex $\vr{Y}{r'}$ for $Y\subseteq \R^n$ finite and $r'>0$ can be obtained up to isomorphism as a simplicial complex $\vr{X}{r}$ with $X\subseteq S^n$ finite and $0<r<\pi$.
The homotopy types of Vietoris--Rips complexes of subsets of $\R^n$ are not fully understood even for $n=2$ (see Question~\ref{ques:planar-VR}), and only get more complicated as $n$ grows.
But, for $X\subseteq S^n$, perhaps restricting attention to $r>r_n$ simplifies the possible homotopy types of $\vr{X}{r}$.
\end{question}

\begin{question}
For fixed $a_1,\ldots,a_{n+1}>0$, let \[E^n=\left\{(x_1,\ldots,x_{n+1}\in \R^n~\Big|~\sum_{i=1}^{n+1}(x_i/a_i)^2=1\right\}\]
be an $n$-dimensional ellipsoid.
What are the homotopy types of $\vr{E^n}{r}$, of course as a function of the parameters $a_i$?

See Example~\ref{ex:ellipse} for partial information in the case $n=2$.

Since an ellipsoid has less symmetries than a sphere, it is conceivable that the homotopy types of $\vr{E^n}{r}$ may be easier to prove than the homotopy types of $\vr{S^n}{r}$.
Indeed, a proof might only require a Morse theory instead of Morse--Bott theory, or alternatively, a Morse--Bott theory that only has to account for fewer symmetries. 
\end{question}



\begin{question}
What are the homotopy types of Vietoris--Rips complexes of the real projective spaces $\RP^n$ and the complex projective spaces $\CP^n$?
We refer the interested reader to~\cite{AdamsHeimPeterson,katz1991rational} for some limited initial information.
See also Corollary~\ref{cor:volume} for how to lower bound the homotopy connectivity of these Vietoris--Rips complexes.
\end{question}

\begin{question}[Question~5 in Section~2 of~\cite{gasarch2017open} and Problem~7.3 in~\cite{MA-FF-AV}]
\label{ques:planar-VR}
Is the Vietoris--Rips complex of any set of points in the plane always homotopy equivalent to a wedge of spheres?

The fundamental group of any such Vietoris--Rips complex is a free group by~\cite{Chambers2010}.
So, the Vietoris--Rips complex of a set of points in the plane cannot be homotopy equivalent to $S^1 \times S^1$, for example.
If the answer to this question is affirmative, then the Vietoris--Rips complex of a set of points in the plane also could not be homotopy equivalent to $S^3 \times S^3$, for example, but so far we do not know how to rule out this possibility.

Not all of the homology of a Vietoris--Rips complex of a set of points of the plane is generated by boundaries of cross-polytopes; see~\cite[Remark~8.5]{AA-VRS1}.

Theorem~6.3 of~\cite{MA-FF-AV} and Theorem~B of~\cite{sipani2024structural} provide evidence towards an affirmative answer to this question.
\end{question}

\begin{question}
Let $(\Z^n,d_1)$ be the integer lattice, equipped with the $L^1$ metric given by $d_1(x,y)=\sum_{i=1}^n|x_i-y_i|$.
Question~3 at~\cite{zaremsky-open} asks if $\vr{(\Z^n,d_1)}{r}$ is contractible when the scale $r$ is sufficiently large, and an affirmative answer is given in~\cite{virk2024contractibility}.
But is $\vr{(\Z^n,d_1)}{r}$ contractible for each $r\ge n$?
\end{question}

\begin{question}
What are the homotopy types of Vietoris--Rips complexes of hypercube graphs?
See~\cite{carlsson2020persistent,adamaszek2022vietoris,shukla2022vietoris,feng2023homotopy,adams2024lower,adams2022v,saleh2024vietoris,bendersky2023connectivity} for recent partial progress on this question.
As we mentioned,~\cite{bendersky2023connectivity} was the inspiration for this paper.
\end{question}


\begin{question}
What are the Urysohn widths of Euclidean balls $B^n$ and spheres $S^n$, and are these values related at all to Vietoris--Rips complexes of spheres?

We refer the reader to~\cite[Section~2.5]{balitskiy2021bounds} for more information on these Urysohn widths.
The value $\mathrm{UW}_{n-1}(B^n)$ is equal to the diameter of an inscribed regular $n$-simplex in $S^n$; see~\cite[Remark~6.10]{akopyan2012borsuk} and~\cite{sitnikov1958rundheit}.
One has $\mathrm{UW}_k(B^n)=\diam(B^n)$ for $k<\frac{n}{2}$, which follows from~\cite{vsvcepin1974problem}.
The value $\mathrm{UW}_k(B^n)$ is not known for $\frac{n}{2} \le k \le n-2$.
See~\cite[Section~2.5]{balitskiy2021bounds} for a description of the (mostly unknown) values of Urysohn widths of spheres, as well.

\end{question}

\section*{Acknowledgements}

HA was funded in part by the Simons Foundation’s Travel Support for Mathematicians program.
JB was funded in part by the Southeast Center for Mathematics and Biology, an NSF-Simons Research Center for Mathematics of Complex Biological Systems, under National Science Foundation Grant No.\ DMS-1764406 and Simons Foundation Grant No.\ 594594.
\v{Z}V was funded in part by the Slovene research agency grants J1-4001 and P1-0292.

\bibliographystyle{plain}
\bibliography{TheConnectivityOfVietorisRipsComplexesOfSpheres}

\begin{thebibliography}{100}

\bibitem{Adamaszek2013}
Micha{\l} Adamaszek.
\newblock Clique complexes and graph powers.
\newblock {\em Israel Journal of Mathematics}, 196(1):295--319, 2013.

\bibitem{AA-VRS1}
Micha{\l} Adamaszek and Henry Adams.
\newblock The {V}ietoris--{R}ips complexes of a circle.
\newblock {\em Pacific Journal of Mathematics}, 290:1--40, 2017.

\bibitem{adamaszek2022vietoris}
Micha{\l} Adamaszek and Henry Adams.
\newblock On {V}ietoris--{R}ips complexes of hypercube graphs.
\newblock {\em Journal of Applied and Computational Topology}, 6:177--192,
  2022.

\bibitem{AAF}
Micha{\l} Adamaszek, Henry Adams, and Florian Frick.
\newblock Metric reconstruction via optimal transport.
\newblock {\em SIAM Journal on Applied Algebra and Geometry}, 2(4):597--619,
  2018.

\bibitem{AAFPP-J}
Micha{\l} Adamaszek, Henry Adams, Florian Frick, Chris Peterson, and Corrine
  Previte-Johnson.
\newblock Nerve complexes of circular arcs.
\newblock {\em Discrete \& Computational Geometry}, 56:251--273, 2016.

\bibitem{AAM}
Micha{\l} Adamaszek, Henry Adams, and Francis Motta.
\newblock Random cyclic dynamical systems.
\newblock {\em Advances in Applied Mathematics}, 83:1--23, 2017.

\bibitem{AAR}
Micha{\l} Adamaszek, Henry Adams, and Samadwara Reddy.
\newblock On {V}ietoris--{R}ips complexes of ellipses.
\newblock {\em Journal of Topology and Analysis}, 11:661--690, 2019.

\bibitem{MA-FF-AV}
Micha{\l} Adamaszek, Florian Frick, and Adrien Vakili.
\newblock On homotopy types of {E}uclidean {R}ips complexes.
\newblock {\em Discrete \& Computational Geometry}, 58(3):526--542, 2017.

\bibitem{GH-BU-VR}
Henry Adams, Johnathan Bush, Nate Clause, Florian Frick, Mario G\'{o}mez,
  Michael Harrison, R.~Amzi Jeffs, Evgeniya Lagoda, Sunhyuk Lim, Facundo
  M\'{e}moli, Michael Moy, Nikola Sadovek, Matt Superdock, Daniel Vargas,
  Qingsong Wang, and Ling Zhou.
\newblock Gromov--{H}ausdorff distances, {B}orsuk--{U}lam theorems, and
  {V}ietoris--{R}ips complexes.
\newblock {\em arXiv preprint arXiv:2301.00246}, 2023.

\bibitem{ABF}
Henry Adams, Johnathan Bush, and Florian Frick.
\newblock Metric thickenings, {B}orsuk--{U}lam theorems, and orbitopes.
\newblock {\em Mathematika}, 66:79--102, 2020.

\bibitem{ABF2}
Henry Adams, Johnathan Bush, and Florian Frick.
\newblock The topology of projective codes and the distribution of zeros of odd
  maps.
\newblock {\em Accepted to appear in Michigan Mathematical Journal}, 2022.

\bibitem{PersistentEquivariantCohomology}
Henry Adams, Aditya De~Saha, Evgeniya Lagoda, Michael Moy, and Nikola Sadovek.
\newblock Persistent equivaraint cohomology.
\newblock Forthcoming, 2024.

\bibitem{HvsGH}
Henry Adams, Florian Frick, Sushovan~Majhi Majhi, and Nicholas McBride.
\newblock Hausdorff vs {G}romov--{H}ausdorff distances.
\newblock {\em arXiv preprint arXiv:2309.16648}, 2023.

\bibitem{HA-FF-ZV}
Henry Adams, Florian Frick, and {\v{Z}}iga Virk.
\newblock Vietoris thickenings and complexes have isomorphic homotopy groups.
\newblock {\em Journal of Applied and Computational Topology}, 7(2):221--241,
  2023.

\bibitem{AdamsHeimPeterson}
Henry Adams, Mark Heim, and Chris Peterson.
\newblock Metric thickenings and group actions.
\newblock {\em Journal of Topology and Analysis}, 14:587--613, 2022.

\bibitem{AMMW}
Henry Adams, Facundo M{\'e}moli, Michael Moy, and Qingsong Wang.
\newblock The persistent topology of optimal transport based metric
  thickenings.
\newblock {\em Algebraic \& Geometric Topology}, 4:393--447, 2024.

\bibitem{adams2022v}
Henry Adams, Samir Shukla, and Anurag Singh.
\newblock \v{C}ech complexes of hypercube graphs.
\newblock {\em arXiv preprint arXiv:2212.05871}, 2022.

\bibitem{adams2024lower}
Henry Adams and {\v{Z}}iga Virk.
\newblock Lower bounds on the homology of {V}ietoris--{R}ips complexes of
  hypercube graphs.
\newblock {\em Bulletin of the Malaysian Mathematical Sciences Society},
  47(72), 2024.

\bibitem{aharoni2000hall}
Ron Aharoni and Penny Haxell.
\newblock Hall's theorem for hypergraphs.
\newblock {\em Journal of Graph Theory}, 35(2):83--88, 2000.

\bibitem{akopyan2012borsuk}
Arseniy Akopyan, Roman Karasev, and Alexey Volovikov.
\newblock Borsuk--{U}lam type theorems for metric spaces.
\newblock {\em arXiv preprint arXiv:1209.1249}, 2012.

\bibitem{attali2022optimal}
Dominique Attali, Hana Dal~Poz Kou{\v{r}}imsk{\'a}, Christopher Fillmore,
  Ishika Ghosh, Andr{\'e} Lieutier, Elizabeth Stephenson, and Mathijs
  Wintraecken.
\newblock Optimal homotopy reconstruction results \`{a}la {N}iyogi, {S}male,
  and {W}einberger for subsets of {E}uclidean spaces and of {R}iemannian
  manifolds.
\newblock {\em arXiv preprint arXiv:2206.10485}, 2022.

\bibitem{balitskiy2021bounds}
Alexey Balitskiy.
\newblock {\em Bounds on {U}rysohn width}.
\newblock PhD thesis, Massachusetts Institute of Technology, 2021.

\bibitem{barmak2023connectivity}
Jonathan~Ariel Barmak.
\newblock Connectivity of ample, conic, and random simplicial complexes.
\newblock {\em International Mathematics Research Notices}, 2023(8):6579--6597,
  2023.

\bibitem{bartholdi2012hodge}
Laurent Bartholdi, Thomas Schick, Nat Smale, and Steve Smale.
\newblock Hodge theory on metric spaces.
\newblock {\em Foundations of Computational Mathematics}, 12(1):1--48, 2012.

\bibitem{bauer2021ripser}
Ulrich Bauer.
\newblock Ripser: efficient computation of {V}ietoris--{R}ips persistence
  barcodes.
\newblock {\em Journal of Applied and Computational Topology}, pages 391--423,
  2021.

\bibitem{bauer2021gromov}
Ulrich Bauer and Fabian Roll.
\newblock Gromov hyperbolicity, geodesic defect, and apparent pairs in
  {V}ietoris--{R}ips filtrations.
\newblock {\em arXiv preprint arXiv:2112.06781}, 2021.

\bibitem{bendersky2023connectivity}
Martin Bendersky and Jelena Grbi\'{c}.
\newblock On the connectivity of the {V}ietoris--{R}ips complex of a hypercube
  graph.
\newblock {\em arXiv preprint arXiv:2311.06407}, 2023.

\bibitem{berestovskii2007uniform}
Valera Berestovskii and Conrad Plaut.
\newblock Uniform universal covers of uniform spaces.
\newblock {\em Topology and its Applications}, 154(8):1748--1777, 2007.

\bibitem{bestvina1997morse}
Mladen Bestvina and Noel Brady.
\newblock Morse theory and finiteness properties of groups.
\newblock {\em Inventiones mathematicae}, 129(3):445--470, 1997.

\bibitem{Bjorner1995}
Anders Bj{\"o}rner.
\newblock Topological methods.
\newblock {\em Handbook of Combinatorics}, 2:1819--1872, 1995.

\bibitem{Borsuk1948}
Karol Borsuk.
\newblock On the imbedding of systems of compacta in simplicial complexes.
\newblock {\em Fundamenta Mathematicae}, 35(1):217--234, 1948.

\bibitem{bott1954nondegenerate}
Raoul Bott.
\newblock Nondegenerate critical manifolds.
\newblock {\em Annals of Mathematics}, pages 248--261, 1954.

\bibitem{bott1982lectures}
Raoul Bott.
\newblock Lectures on {M}orse theory, old and new.
\newblock {\em Bulletin of the American Mathematical Society}, 7(2):331--358,
  1982.

\bibitem{bott1988morse}
Raoul Bott.
\newblock Morse theory indomitable.
\newblock {\em Publications Math{\'e}matiques de l'IH{\'E}S}, 68:99--114, 1988.

\bibitem{bridson2011metric}
Martin~R Bridson and Andr{\'e} Haefliger.
\newblock {\em Metric spaces of non-positive curvature}, volume 319.
\newblock Springer Science \& Business Media, 2011.

\bibitem{brodskiy2013rips}
N~Brodskiy, J~Dydak, B~Labuz, and A~Mitra.
\newblock Rips complexes and covers in the uniform category.
\newblock {\em Houston Journal of Mathematics}, 39:667--699, 2013.

\bibitem{BushMasters}
Johnathan Bush.
\newblock {V}ietoris--{R}ips thickenings of the circle and centrally-symmetric
  orbitopes.
\newblock Master's thesis, Colorado State University, 2018.

\bibitem{BushThesis}
Johnathan Bush.
\newblock {\em Topological, geometric, and combinatorial aspects of metric
  thickenings}.
\newblock PhD thesis, Colorado State University, 2021.

\bibitem{Carlsson2009}
Gunnar Carlsson.
\newblock Topology and data.
\newblock {\em Bulletin of the American Mathematical Society}, 46(2):255--308,
  2009.

\bibitem{carlsson2020persistent}
Gunnar Carlsson and Benjamin Filippenko.
\newblock Persistent homology of the sum metric.
\newblock {\em Journal of Pure and Applied Algebra}, 224(5):106244, 2020.

\bibitem{cencelj2012combinatorial}
M~Cencelj, J~Dydak, A~Vavpeti{\v{c}}, and {\v{Z}}~Virk.
\newblock A combinatorial approach to coarse geometry.
\newblock {\em Topology and its Applications}, 159(3):646--658, 2012.

\bibitem{Chambers2010}
Erin~W. Chambers, Vin de~Silva, Jeff Erickson, and Robert Ghrist.
\newblock Vietoris--{R}ips complexes of planar point sets.
\newblock {\em Discrete \& Computational Geometry}, 44(1):75--90, 2010.

\bibitem{chazal2009gromov}
Fr{\'e}d{\'e}ric Chazal, David Cohen-Steiner, Leonidas~J Guibas, Facundo
  M{\'e}moli, and Steve~Y Oudot.
\newblock Gromov--{H}ausdorff stable signatures for shapes using persistence.
\newblock In {\em Computer Graphics Forum}, volume~28, pages 1393--1403, 2009.

\bibitem{ChazalDeSilvaOudot2014}
Fr{\'e}d{\'e}ric Chazal, Vin de~Silva, and Steve Oudot.
\newblock Persistence stability for geometric complexes.
\newblock {\em Geometriae Dedicata}, 174:193--214, 2014.

\bibitem{ChazalOudot2008}
Fr{\'e}d{\'e}ric Chazal and Steve Oudot.
\newblock Towards persistence-based reconstruction in {E}uclidean spaces.
\newblock In {\em Proceedings of the 24th Symposium on Computational Geometry},
  pages 232--241. ACM, 2008.

\bibitem{chudnovsky2000systems}
Maria Chudnovsky.
\newblock Systems of disjoint representatives.
\newblock Master's thesis, 2000.

\bibitem{cohen2007stability}
David Cohen-Steiner, Herbert Edelsbrunner, and John Harer.
\newblock Stability of persistence diagrams.
\newblock {\em Discrete \& Computational Geometry}, 37(1):103--120, 2007.

\bibitem{conant2014discrete}
Jim Conant, Victoria Curnutte, Corey Jones, Conrad Plaut, Kristen Pueschel,
  Maria Walpole, and Jay Wilkins.
\newblock Discrete homotopy theory and critical values of metric spaces.
\newblock {\em Fundamenta Mathematicae}, 227:97--128, 2014.

\bibitem{crabb2023borsuk}
Michael~C Crabb.
\newblock On {B}orsuk--{U}lam theorems and convex sets.
\newblock {\em Mathematika}, 69(2):366--370, 2023.

\bibitem{EdelsbrunnerHarer}
Herbert Edelsbrunner and John~L Harer.
\newblock {\em Computational Topology: An Introduction}.
\newblock American Mathematical Society, Providence, 2010.

\bibitem{edelsbrunner2000topological}
Herbert Edelsbrunner, David Letscher, and Afra Zomorodian.
\newblock Topological persistence and simplification.
\newblock In {\em Proceedings 41st Symposium on Foundations of Computer
  Science}, pages 454--463. IEEE, 2000.

\bibitem{farber2023large}
Michael Farber.
\newblock Large simplicial complexes: universality, randomness, and ampleness.
\newblock {\em Journal of Applied and Computational Topology}, pages 1--24,
  2023.

\bibitem{feng2023homotopy}
Ziqin Feng.
\newblock Homotopy types of {V}ietoris--{R}ips complexes of hypercube graphs.
\newblock {\em arXiv preprint arXiv:2305.07084}, 2023.

\bibitem{forman2002user}
Robin Forman.
\newblock A user's guide to discrete {M}orse theory.
\newblock {\em S{\'e}minaire Lotharingien de Combinatoire}, 48:B48c--35, 2002.

\bibitem{fowler2002circle}
Patrick~W Fowler, Tibor Tarnai, and Zsolt G{\'a}sp{\'a}r.
\newblock From circle packing to covering on a sphere with antipodal
  constraints.
\newblock {\em Proceedings of the Royal Society of London. Series A:
  Mathematical, Physical and Engineering Sciences}, 458(2025):2275--2287, 2002.

\bibitem{gasarch2017open}
William Gasarch, Brittany~Terese Fasy, and Bei Wang.
\newblock Open problems in computational topology.
\newblock {\em ACM SIGACT News}, 48(3):32--36, 2017.

\bibitem{gasparovic2018complete}
Ellen Gasparovic, Maria Gommel, Emilie Purvine, Radmila Sazdanovic, Bei Wang,
  Yusu Wang, and Lori Ziegelmeier.
\newblock A complete characterization of the one-dimensional intrinsic \v{C}ech
  persistence diagrams for metric graphs.
\newblock In {\em Research in Computational Topology}, pages 33--56. Springer,
  2018.

\bibitem{gillespie2022homological}
Patrick Gillespie.
\newblock A homological nerve theorem for open covers.
\newblock {\em arXiv preprint arXiv:2210.00388}, 2022.

\bibitem{gillespie2023vietoris}
Patrick Gillespie.
\newblock Vietoris thickenings and complexes are weakly homotopy equivalent.
\newblock {\em arXiv preprint arXiv:2303.01019}, 2023.

\bibitem{GlissePritam2022}
Marc Glisse and Siddharth Pritam.
\newblock Swap, shift and trim to edge collapse a filtration.
\newblock In {\em Proceedings of the 38th Symposium on Computational Geometry},
  pages 44:1--44:15, 2022.

\bibitem{gorivcan2023critical}
Peter Gori{\v{c}}an and {\v{Z}}iga Virk.
\newblock Critical edges in {R}ips complexes and persistence.
\newblock {\em Mediterranean Journal of Mathematics}, 20(6):326, 2023.

\bibitem{gromov1983filling}
Mikhael Gromov.
\newblock Filling {R}iemannian manifolds.
\newblock {\em Journal of Differential Geometry}, 18(1):1--147, 1983.

\bibitem{Gromov}
Mikhail Gromov.
\newblock Geometric group theory, volume 2: {A}symptotic invariants of infinite
  groups.
\newblock {\em London Mathematical Society Lecture Notes}, 182:1--295, 1993.

\bibitem{gromov2007metric}
Mikhail Gromov.
\newblock {\em Metric structures for {R}iemannian and non-{R}iemannian spaces}.
\newblock Springer Science \& Business Media, 2007.

\bibitem{Hausmann1995}
Jean-Claude Hausmann.
\newblock On the {V}ietoris--{R}ips complexes and a cohomology theory for
  metric spaces.
\newblock {\em Annals of Mathematics Studies}, 138:175--188, 1995.

\bibitem{horvath2021strongly}
{\'A}kos~G Horv{\'a}th.
\newblock Strongly self-dual polytopes and distance graphs in the unit sphere.
\newblock {\em Acta Mathematica Hungarica}, 163(2):640--651, 2021.

\bibitem{jucovic1960some}
E~Jucovic.
\newblock Some coverings of a spherical surface with equal circles (in
  {S}lovakian).
\newblock {\em Mat.-Fyz. Casopis. Slovensk. Akad. Vied}, 10:99--104, 1960.

\bibitem{Kahle2009}
Matthew Kahle.
\newblock Topology of random clique complexes.
\newblock {\em Discrete Mathematics}, 309(6):1658--1671, 2009.

\bibitem{katz1983filling}
Mikhail Katz.
\newblock The filling radius of two-point homogeneous spaces.
\newblock {\em Journal of Differential Geometry}, 18(3):505--511, 1983.

\bibitem{katz1989diameter}
Mikhail Katz.
\newblock Diameter-extremal subsets of spheres.
\newblock {\em Discrete \& Computational Geometry}, 4(2):117--137, 1989.

\bibitem{katz9filling}
Mikhail Katz.
\newblock The filling radius of homogeneous manifolds.
\newblock {\em S{\'e}minaire de th{\'e}orie spectrale et g{\'e}om{\'e}trie},
  9:103--109, 1990-1991.

\bibitem{katz1991neighborhoods}
Mikhail Katz.
\newblock On neighborhoods of the {K}uratowski imbedding beyond the first
  extremum of the diameter functional.
\newblock {\em Fundamenta Mathematicae}, 137(3):161--175, 1991.

\bibitem{katz1991rational}
Mikhail Katz.
\newblock The rational filling radius of complex projective space.
\newblock {\em Topology and its Applications}, 42(3):201--215, 1991.

\bibitem{katz2023extremal}
Mikhail Katz, Facundo M{\'e}moli, and Qingsong Wang.
\newblock Extremal spherical polytopes and {B}orsuk's conjecture.
\newblock {\em arXiv preprint arXiv:2301.13076}, 2023.

\bibitem{komendarczyk2024topological}
Rafal Komendarczyk, Sushovan Majhi, and Will Tran.
\newblock Topological stability and {L}atschev-type reconstruction theorems for
  {CAT}($\kappa$) spaces.
\newblock {\em arXiv preprint arXiv:2406.04259}, 2024.

\bibitem{Latschev2001}
Janko Latschev.
\newblock Vietoris--{R}ips complexes of metric spaces near a closed
  {R}iemannian manifold.
\newblock {\em Archiv der Mathematik}, 77(6):522--528, 2001.

\bibitem{lee2007nonlinear}
John~A Lee and Michel Verleysen.
\newblock {\em Nonlinear dimensionality reduction}, volume~1.
\newblock Springer, 2007.

\bibitem{lefschetz1942algebraic}
Solomon Lefschetz.
\newblock {\em Algebraic topology}, volume~27.
\newblock American Mathematical Society, 1942.

\bibitem{lim2022vietoris}
Sunhyuk Lim, Facundo M{\'e}moli, and Osman~B Okutan.
\newblock Vietoris--{R}ips persistent homology, injective metric spaces, and
  the filling radius.
\newblock {\em Accepted to appear in Algebraic \& Geometric Topology}, 2022.

\bibitem{lim2023strange}
Uzu Lim.
\newblock Strange random topology of the circle.
\newblock {\em arXiv preprint arXiv:2305.16270}, 2023.

\bibitem{lovasz1983self}
L{\'a}sl{\'o} Lov{\'a}sz.
\newblock Self-dual polytopes and the chromatic number of distance graphs on
  the sphere.
\newblock {\em Acta Scientiarum Mathematicarum}, 45(1-4):317--323, 1983.

\bibitem{majhi2022vietoris}
Sushovan Majhi.
\newblock Vietoris--{R}ips complexes of metric spaces near a metric graph.
\newblock {\em arXiv preprint arXiv:2204.14234}, 2022.

\bibitem{majhi2023demystifying}
Sushovan Majhi.
\newblock Demystifying {L}atschev's theorem: {M}anifold reconstruction from
  noisy data.
\newblock {\em arXiv preprint arXiv:2305.17288}, 2023.

\bibitem{malyutin2023neighboring}
Andrei~V Malyutin and Oleg~R Musin.
\newblock Neighboring mapping points theorem.
\newblock {\em Algebraic \& Geometric Topology}, 23(7):3043--3070, 2023.

\bibitem{Matouvsek2008}
Ji{\v{r}}{\'\i} Matou{\v{s}}ek.
\newblock {LC} reductions yield isomorphic simplicial complexes.
\newblock {\em Contributions to Discrete Mathematics}, 3(2), 2008.

\bibitem{matousek2003using}
Ji\v{r}\'{i} Matou\v{s}ek.
\newblock {\em Using the Borsuk--{U}lam theorem}.
\newblock Springer Science \& Business Media, 2008.

\bibitem{meshulam2001clique}
Roy Meshulam.
\newblock The clique complex and hypergraph matching.
\newblock {\em Combinatorica}, 21(1):89--94, 2001.

\bibitem{meshulam2003domination}
Roy Meshulam.
\newblock Domination numbers and homology.
\newblock {\em Journal of Combinatorial Theory, Series A}, 102(2):321--330,
  2003.

\bibitem{milnor1963morse}
John~Willard Milnor.
\newblock {\em Morse theory}.
\newblock Princeton university press, 1963.

\bibitem{MirthThesis}
Joshua Mirth.
\newblock {\em Vietoris--{R}ips metric thickenings and {W}asserstein spaces}.
\newblock PhD thesis, Colorado State University, 2020.

\bibitem{MoyMasters}
Michael Moy.
\newblock Persistence stability for metric thickenings.
\newblock Master's thesis, Colorado State University, 2021.

\bibitem{moyVRmS1}
Michael Moy.
\newblock Vietoris--{R}ips metric thickenings of the circle.
\newblock {\em Accepted to appear in Journal of Applied and Computational
  Topology}, 2023.

\bibitem{niyogi2008finding}
Partha Niyogi, Stephen Smale, and Shmuel Weinberger.
\newblock Finding the homology of submanifolds with high confidence from random
  samples.
\newblock {\em Discrete \& Computational Geometry}, 39(1):419--441, 2008.

\bibitem{okutan2019persistence}
Osman~Berat Okutan.
\newblock {\em Persistence, metric invariants, and simplification}.
\newblock PhD thesis, The Ohio State University, 2019.

\bibitem{plaut2013discrete}
Conrad Plaut and Jay Wilkins.
\newblock Discrete homotopies and the fundamental group.
\newblock {\em Advances in Mathematics}, 232(1):271--294, 2013.

\bibitem{robins2000computational}
Vanessa Robins.
\newblock {\em Computational topology at multiple resolutions: {F}oundations
  and applications to fractals and dynamics}.
\newblock PhD thesis, University of Colorado at Boulder, 2000.

\bibitem{roe1993coarse}
John Roe.
\newblock {\em Coarse cohomology and index theory on complete {R}iemannian
  manifolds}, volume 497.
\newblock American Mathematical Soc., 1993.

\bibitem{sakai2013geometric}
Katsuro Sakai.
\newblock {\em Geometric aspects of general topology}.
\newblock Springer, 2013.

\bibitem{saleh2024vietoris}
Nada Saleh, Thomas~Titz Mite, and Stefan Witzel.
\newblock {V}ietoris--{R}ips complexes of platonic solids.
\newblock {\em Innovations in Incidence Geometry}, 21(1):17--34, 2024.

\bibitem{vsvcepin1974problem}
EV~{\v{S}}{\v{c}}epin.
\newblock On a problem of {LA} {T}umarkin.
\newblock In {\em Soviet Mathematics Doklady}, volume~15, pages 1024--1026,
  1974.

\bibitem{schutte1955uberdeckungen}
Kurt Sch{\"u}tte.
\newblock {\"U}berdeckungen der kugel mit h{\"o}chstens acht kreisen.
\newblock {\em Mathematische Annalen}, 129(1):181--186, 1955.

\bibitem{shukla2022vietoris}
Samir Shukla.
\newblock On {V}ietoris--{R}ips complexes (with scale 3) of hypercube graphs.
\newblock {\em arXiv preprint arXiv:2202.02756}, 2022.

\bibitem{sipani2024structural}
Vinay Sipani and Ramesh Kasilingam.
\newblock Structural characterizations and obstructions to planar-{R}ips
  complexes.
\newblock {\em arXiv preprint arXiv:2406.01082}, 2024.

\bibitem{sitnikov1958rundheit}
Kirill Sitnikov.
\newblock {\em {\"U}ber die rundheit der kugel}.
\newblock Vandenhoeck \& Ruprecht, 1958.

\bibitem{tarnai1985covering}
Tibor Tarnai and Zsolt G{\'a}sp{\'a}r.
\newblock Covering the sphere with equal circles.
\newblock In {\em Intuitive Geometry, Colloquia Mathematica Societatis Jdnos
  Bolyai}, volume~48, pages 545--550, 1985.

\bibitem{tarnai1986covering}
Tibor Tarnai and Zsolt G{\'a}sp{\'a}r.
\newblock Covering the sphere with 11 equal circles.
\newblock {\em Elemente der Mathematik}, 41:35--38, 1986.

\bibitem{tarnai1991covering}
Tibor Tarnai and Zsolt G{\'a}sp{\'a}r.
\newblock Covering a sphere by equal circles, and the rigidity of its graph.
\newblock In {\em Mathematical Proceedings of the Cambridge Philosophical
  Society}, volume 110, pages 71--89. Cambridge University Press, 1991.

\bibitem{tom1987transgroups}
Tammo tom Dieck.
\newblock {\em Transformation groups}, volume~8 of {\em De Gruyter Studies in
  Mathematics}.
\newblock Walter de Gruyter \& Co., Berlin, 1987.

\bibitem{toth1969}
G~Fejes T{\'o}th.
\newblock Kreis{\"u}berdeckungen der sphare.
\newblock {\em Studia Sci.\ Math.\ Hungar}, 4:225--247, 1969.

\bibitem{toth1943covering}
L~Fejes T{\'o}th.
\newblock On covering a spherical surface with equal spherical caps (in
  {H}ungarian).
\newblock {\em Matematikai {\'e}s Fizikai Lapolc}, 50, 1943.

\bibitem{vershik2013long}
Anatoly~Moiseevich Vershik.
\newblock {Long history of the Monge--Kantorovich transportation problem}.
\newblock {\em The Mathematical Intelligencer}, 35(4):1--9, 2013.

\bibitem{Vietoris27}
Leopold Vietoris.
\newblock {{\"U}ber den h{\"o}heren Zusammenhang kompakter R{\"a}ume und eine
  Klasse von zusammenhangstreuen Abbildungen}.
\newblock {\em Mathematische Annalen}, 97(1):454--472, 1927.

\bibitem{villani2003topics}
C{\'e}dric Villani.
\newblock {\em Topics in optimal transportation}.
\newblock Number~58. American Mathematical Soc., Providence, 2003.

\bibitem{villani2008optimal}
C{\'e}dric Villani.
\newblock {\em Optimal transport: Old and new}, volume 338.
\newblock Springer, Berlin, 2008.

\bibitem{virk2017approximations}
{\v{Z}}iga Virk.
\newblock Approximations of 1-dimensional intrinsic persistence of geodesic
  spaces and their stability.
\newblock {\em Revista Matem\'{a}tica Complutense}, 32:195--213, 2019.

\bibitem{virk20201}
{\v{Z}}iga Virk.
\newblock 1-dimensional intrinsic persistence of geodesic spaces.
\newblock {\em Journal of Topology and Analysis}, 12:169--207, 2020.

\bibitem{virk2021counter}
{\v{Z}}iga Virk.
\newblock A counter-example to {H}ausmann's conjecture.
\newblock {\em Foundations of Computational Mathematics}, 2021.

\bibitem{virkSelective}
{\v{Z}}iga Virk.
\newblock Persistent homology with selective {R}ips complexes detects geodesic
  circles.
\newblock {\em arXiv preprint arXiv:2108.07460}, 2021.

\bibitem{virk2021rips}
{\v{Z}}iga Virk.
\newblock Rips complexes as nerves and a functorial {D}owker-nerve diagram.
\newblock {\em Mediterranean Journal of Mathematics}, 18(2):1--24, 2021.

\bibitem{virk2022contractions}
{\v{Z}}iga Virk.
\newblock Contractions in persistence and metric graphs.
\newblock {\em Bulletin of the Malaysian Mathematical Sciences Society,
  \emph{\url{https://doi.org/10.1007/s40840-022-01368-z}}}, 2022.

\bibitem{virk2021footprints}
{\v{Z}}iga Virk.
\newblock Footprints of geodesics in persistent homology.
\newblock {\em Mediterranean Journal of Mathematics}, 19:160, 2022.

\bibitem{virk2024contractibility}
{\v{Z}}iga Virk.
\newblock Contractibility of the {R}ips complexes of integer lattices via local
  domination.
\newblock {\em arXiv preprint arXiv:2405.09134}, 2024.

\bibitem{Hocolim}
Volkmar Welker, G{\"u}nter~M Ziegler, and \v{Z}ivaljevi{\'c} Rade~T.
\newblock Homotopy colimits --- {C}omparison lemmas for combinatorial
  applications.
\newblock {\em Journal f{\"u}r die Reine und Angewandte Mathematik (Crelles
  Journal)}, 509:117--149, 1999.

\bibitem{WilkinsThesis}
Leonard Wilkins.
\newblock {\em Discrete Geometric Homotopy Theory and Critical Values of Metric
  Spaces}.
\newblock PhD thesis, University of Tennessee, 2011.

\bibitem{zaremsky-open}
Matthew~CB Zaremsky.
\newblock Some open problems.
\newblock \url{https://www.albany.edu/~mz498674/open_problems.pdf}.

\bibitem{zaremsky2022bestvina}
Matthew~CB Zaremsky.
\newblock Bestvina--{B}rady discrete {M}orse theory and {V}ietoris--{R}ips
  complexes.
\newblock {\em American Journal of Mathematics}, 144(5):1177--1200, 2022.

\bibitem{zhang2020gpu}
Simon Zhang, Mengbai Xiao, and Hao Wang.
\newblock {G}{P}{U}-accelerated computation of {V}ietoris--{R}ips persistence
  barcodes.
\newblock {\em arXiv preprint arXiv:2003.07989}, 2020.

\end{thebibliography}

\end{document}